\documentclass[11pt]{article}

\usepackage{amsmath,amssymb,amsfonts,epsf,epsfig}
\usepackage{color,url,hyperref}
\usepackage[normalem]{ulem}

\usepackage{euscript,graphicx}

\newtheorem{theorem}{Theorem}[section]

\newtheorem{construction}[theorem]{Construction}

\newcommand{\Aut}{{\rm Aut}}
\newcommand{\DW}{{\rm DW}}
\newcommand{\PS}{{\rm PS}}
\newcommand{\CS}{{\rm CS}}
\newcommand{\MPS}{{\rm MPS}}
\newcommand{\BC}{{\rm BC}}
\newcommand{\W}{{\rm W}}
\newcommand{\SDD}{{\rm SDD}}
\newcommand{\SoP}{{\rm SoP}}
\newcommand{\SSS}{{\rm SS}}

\newcommand{\Br}{{\rm Br}}

\newcommand{\Prr}{{\rm Pr}}
\newcommand{\LoPr}{{\rm LoPr}}
\newcommand{\KE}{{\rm KE}}
\newcommand{\Curtain}{{\rm Curtain}}
\newcommand{\WH}{{\rm WH}}
\newcommand{\MSY}{{\rm MSY}}
\newcommand{\MSZ}{{\rm MSZ}}
\newcommand{\CPM}{{\rm CPM}}
\newcommand{\DG}{{\rm DG}}
\newcommand{\RW}{{\rm R}}
\newcommand{\TAG}{{\rm TAG}}
\newcommand{\HC}{{\rm HC}}
\newcommand{\MC}{{\rm MC}}
\newcommand{\AMC}{{\rm AMC}}

\newcommand{\AffLR}{{\rm AffLR}}
\newcommand{\Att}{{\rm Att}}

\newcommand{\AT}{{\rm AT}}
\newcommand{\HT}{{\rm HT}}
\newcommand{\PC}{{\rm PX}}
\newcommand{\PX}{{\rm PX}}
\newcommand{\MG}{{\rm MG}}
\newcommand{\PPM}{{\rm PPM}}
\newcommand{\XI}{{\rm XI}}

\newcommand{\ProjLR}{{\rm ProjLR}}
\newcommand{\BGCG}{{\rm BGCG}}

\renewcommand{\mod}{\hbox{{\rm mod}}\, }

\newcommand{\Cos}{\hbox{{\rm Cos}}}

\newcommand{\ZZ}{\mathbb{Z}}

\newcommand{\PP}{\mathbb{P}}

\newcommand{\Cay}{\hbox{{\rm Cay}}}

\newcommand{\D}{{\mathcal{D}}}

\newcommand{\C}{{\mathcal{C}}}

\newcommand{\A}{{\mathcal{A}}}
\newcommand{\R}{{\mathcal{R}}}
\newcommand{\G}{{\mathcal{G}}}

\newcommand{\M}{{\mathcal{M}}}

\newcommand{\V}{{\mathcal{V}}}
\newcommand{\E}{{\mathcal{E}}}

\newcommand{\half}{\frac{1}{2}}

\begin{document}

\begin{center}
{\bf\large 
Recipes for Edge-Transitive Tetravalent Graphs
}
\end{center}

\bigskip\noindent
\begin{center}
{\sc Steve Wilson} \\
{\small Department of Mathematics and Statistics,}\\
{\small  Northern Arizona University,}\\
{\small Box 5717, Flagstaff, AZ 86011, USA}\\
{\tt stephen.wilson@nau.edu}\\
\end{center}

\bigskip\noindent
\begin{center}
{\sc Primo\v z Poto\v cnik\footnote{%
This author gratefully acknowledges the
    support of the US Department of State and
    the Fulbright Scholar Program who sponsored his visit
    to Northern Arizona University in spring 2004.}} \\
{\small University of Ljubljana,}\\
{\small Faculty of Mathematics and Physics}\\
{\small Jadranska 19, SI-1000 Ljubljana, Slovenia;}\\
{\tt primoz.potocnik@fmf.uni-lj.si}\\
\end{center}
\bigskip

\vfill

\noindent
Proposed running head: {\bf Tetravalent Recipes}

\bigskip
edit date: 14 August 2016

print date: \today
\bigskip

\noindent
The correspondence should be addressed to:

\medskip

\noindent
{\sc Steve Wilson} \\
{\small Department of Mathematics and Statistics,}\\
{\small  Northern Arizona University,}\\
{\small Box 5717, Flagstaff, AZ 86011, USA}\\
{\tt stephen.wilson@nau.edu}\\

\vfill

\hbox{ }

\bigskip

\newpage


\section{The Census}
\label{sc:Census}

This paper is to accompany the Census of Edge-Transitive Tetravalent Graphs, available at
\begin{center}
	 \href{http://jan.ucc.nau.edu/~swilson/C4FullSite/index.html}{http://jan.ucc.nau.edu/$\sim$swilson/C4FullSite/index.html},
	\end{center}
which is a collection of all known edge-transitive graphs of valence 4 up to  512 vertices.  

	The Census contains information for each graph. This information includes parameters such as group order, diameter, girth etc., all known constructions, relations to other graphs in the Census and intersting substructures such as colorings, cycle structures, and dissections.
 
We try to present most graphs  as members of one or more parameterized families, and one purpose of this paper is to gather together, here in one place, descriptions of each of these families, to show how each is constructed, what the history of each is and how one family is related to another.  We also discuss in this paper the theory and techniques behind computer searches leading to many entries in the Census.

	We should point out that similar censi exist for edge transitive graphs of valence 3 \cite{MC3,CMMP}. 
Unlike our census, these censi are complete in the sense that they contain all the graphs up to a given order. The method used
in these papers relies on the fact that in the case of prime valence, the order of the automorphism group can be bounded
be a linear function of the order of the graph, making exhaustive computer searches possible -- see \cite{ConDob} for details.

Even though our census is not proved to be complete, it is complete in some segments. In particular,
the census contains all dart-transitive tetravalent graphs up to 512 vertices (see \cite{CubicCensus,CubicCensusSite})
and all $\frac{1}{2}$-arc-transitive tetravalent  graphs up to 512 vertices (see \cite{HATcensus,HATcensusSite}). Therefore, if a
graph is missing from our census, then it is semisymmetric (see Section~\ref{sc:defs}).

\section{Basic Notions}
\label{sc:defs}

	A {\em graph} is an ordered pair $\Gamma = (\V, \E)$, where $\V$ is an arbitrary set of things called vertices, and $\E$ is a collection of subsets of $\V$ of size two; these are called {\em edges}.  We let $\V(\Gamma) = \V$ and $\E(\Gamma) = \E$ in this case. If $e = \{u, v\} \in \E$, we say $u$ is a {\em neighbor} of  $v$, that $u$ and $v$ are {\em adjacent}, and that $u$ is {\em incident} with $e$ and vice versa.  A {\em dart} or {\em directed edge} is an ordered pair $(u,v)$ where $\{u, v\} \in \E$.  Let $\D(\Gamma)$ be the set of darts of $\Gamma$. The {\em valence} or {\em degree} of a vertex $v$ is the number of edges to which $v$ belongs.  A graph is {\em regular} provided that every vertex has the same valence, and then we refer to that as the valence of the graph.

A {\em digraph} is an ordered pair $\Delta = (\V, \E)$, where $\V$ is an arbitrary set of things called vertices, and $\E$ is a collection of ordered pairs of distinct elements of $\V$.  We think of the pair $(u,v)$ as being an edge directed from $u$ to $v$.  An {\em orientation} is a digraph in which for all $u, v \in\V$, if $(u,v)\in \E$ then $(v,u)\notin \E$.

	A {\em symmetry}, or  {\em automorphism}, of a graph or a digraph $\Gamma$ is a permutation of $\V$ which preserves $\E$.
 If $v\in\V(\Gamma)$ and $g$ is a symmetry of
$\Gamma$, then we denote the image of $v$ under $\sigma$ by $v\sigma$, and if $\rho$ is also a symmetry of $\Gamma$,
 then the product $\sigma$ is a symmetry that maps $v$ to
$(v\sigma)\rho$.
 Together with this product, the set of symmetries of $\Gamma$ forms a group, denoted $\Aut(\Gamma)$.  We are interested in those graphs for which 
$G = \Aut(\Gamma)$ is big enough to be transitive on $\E$.  Such a graph is called  {\em edge-transitive}.  Within the class of edge-transitive graphs of a given valence, there are three varieties:  

\begin{enumerate}
\item[(1)]  A graph is  {\em symmetric} or  {\em dart-transitive} provided that $G$ is transitive on $\D=\D(\Gamma)$.
\item[(2)] A graph is  {\em $\frac{1}{2}$-arc-transitive} provided that $G$ is transitive on $\E$ and on $\V$, but not on $\D$.  A $\frac{1}{2}$-arc-transitive graph must have even valence \cite{tutte}. The $G$-orbit of one dart is then an orientation $\Delta$ of $\Gamma$ such that every vertex has $k$ in-neighbours and $k$ out-neighbors, where $2k$ is the valence of $\Gamma$.  The symmetry group $\Aut(\Delta)$ is then transitive on vertices and on edges.  We call such a $\Delta$  a {\em semitransitive orientation} and we say that a graph which has such an orientation is {\em semitransitive} \cite{STG}.  
\item[(3)] Finally, $\Gamma$ is  {\em semisymmetric} provided that $G$ is transitive on $\E$ but not on $\D$ and not on $\V$.  In this case, the graph must be bipartite, with each edge having one vertex from each class.
More generally, we say that a  graph $\Gamma$ is {\em bi-transitive} provided that $\Gamma$ is bipartite, and its group of color-preserving symmetries is transitive on edges (and so on vertices of each color);
a bi-transitive graph is thus either semisymmetric of dart-transitive.
\end{enumerate}

	There is a fourth important kind of symmetricity that a tetravalent graph might have, in which $\Aut(\Gamma)$ is transitive on vertices but has two orbits on edges, and satisfying certain other conditions.  These are {\em LR structures}, introduced and defined in Section \ref{sc:Cycle Decompositions} but referred to in several places before that.  These graphs, though not edge-transitive themselves, can be used directly to construct semisymmetric graphs.

	For $\sigma\in\Aut(\Gamma)$, if there is a vertex $v\in\V$ such that $v\sigma$ is adjacent to $v$ and  $v\sigma^2\ne v$
we call $\sigma$ a {\em shunt} and
  then $\sigma$ induces a directed cycle $[v, v\sigma, v\sigma^2, \ldots, v\sigma^m]$, with $v\sigma^{m+1} = v$, called a {\em consistent} cycle.
 The remarkable theorem of Biggs and Conway \cite{BC,ConCyc} says that if  $\Gamma$  is dart-transitive and  regular of degree $d$, then there are exactly $d-1$ orbits of consistent cycles.  A later result \cite{ConCycHAT} shows that a  $\frac{1}{2}$-arc-transitive graph of degree $2e$ must have exactly $e$ orbits of consistent cycles.  For tetravalent graphs, the dart-transitive ones have 3 orbits of consistant cycles and the  $\frac{1}{2}$-arc-transitive ones have two such orbits.  An LR structure has 1 or 2 orbits depending on whether it is self-dual or not.

\section{Computer generated lists of graphs}
\label{sc:comp}
\subsection{Semisymmetric graphs arising from amalgams of index $(4,4)$}

Let $L$ and $R$ be two finite groups intersecting in a common subgroup $B$ and assume that no non-trivial subgroup of $B$ is normal in both $L$ and $R$.
Then the triple $(L,B,R)$ is called an {\em amalgam}. For example, if we let $L=A_4$, the alternating group of degree $4$, $B \cong C_3$, viewed as a point-stabiliser in 
$L$, and $R\cong C_{12}$ containing $B$ as a subgroup of index $4$, then $(L,B,R)$ is an amalgam.

If $G$ is a group that contains both $L$ and $R$ and is generated by them, then $G$ is called a {\em completion}
of the amalgam. It is not too difficult to see that there exists a completion that is universal in the sense that every other
completion is isomorphic to a quotient  thereof. This universal completion is sometimes called the {\em free product of 
$L$ and $R$ amalgamated over $B$} (usually denoted by $L*_BR$) and can be constructed by merging together 
(disjoint) presentations of $L$ and $R$ and adding relations that identify copies of the same element of $B$ in both $L$ and $R$.
For example, if the amalgam $(L,B,R)$ is as above, then we can write 
$$
L=\langle x,y,b|x^2,y^2,[x,y],b^3,x^by,y^bxy\rangle, \quad R=\langle z|z^{12} \rangle,
$$
yielding
$$
L*_B R = \langle x,y,b,z|x^2,y^2,[x,y],b^3,z^{12},x^by,y^bxy, z^4=b \rangle.
$$
Completions of a given amalgam $(L,B,R)$ up to a given order, say $M$, can be computed using a {\sc LowIndexNormalSubgroups} routine,
developed by Firth and Holt \cite{FH} and  implemented in {\sc Magma} \cite{magma}.

Given a completion $G$ of an amalgam $(L,B,R)$, one can construct a bipartite graph, called {\em the graph of the completion},
with white and black vertices being the cosets of $L$ and $R$ in $G$, respectively, and two cosets $Lg$ and $Rh$ adjacent whenever they intersect. 
Note that  white (black) vertices are of valence $[B:L|]$  ($[B:R]$, respectively). In particular, if $B$ is of index $4$ in both $L$ and $R$, then the graph is 
tetravalent. We shall say in this case that the amalgam is of index $(4,4)$.

The group $G$ acts by right multiplication faithfully  as an edge-transitive group of automorphism of the graph and so the graph of a completion of an amalgam 
always admits an edge- but not vertex-transitive group of automorphisms, and so the graph is bi-transitive 
(and thus either dart-transitive or semisymmetric). 
This now gives us a good strategy for constructing tetravalent semisymmetric graphs of order at most $M$:
 Choose your favorite  amalgams $(L,B,R)$ of index $(4,4)$, find their completions up to order $2M|B|$ and construct the corresponding graphs.

We have done this for several amalgams of index $(4,4)$ and the resulting graphs appear in the census under the name $\SSS$.
  The graph $\SSS[n,i]$ is the $i$-th graph in the list of semisymmetric graphs of order $n$. These graphs are available in magma code at \cite{HATcensusSite}.

\subsection{Dart-transitive graphs from amalgams}

If $\Gamma$ is a tetravalent dart-transitive graph, then its {\em subdivision}, obtained from $\Gamma$ by inserting
a vertex of valence $2$ on each edge, is edge-transitive but not vertex-transitive. This process is reversible,
by removing vertices of degree $2$ in a bi-transitive graph of valence $(4,2)$ one obtains a tetravalent dart-transitive graph.

In the spirit of the previous section,
each such graph can be obtained from an amalgam of index $(4,2)$. Amalgams of index $(4,2)$ were fully classified in
\cite{djo,AT2,weiss}, however, unlike in the case of amalgams of index $(3,2)$, giving rise to cubic dart-transitive graphs,
the number of these amalgams is infinite. This fact, together with existence of relatively small
tetravalent dart-transitive graphs with very large automorphism groups  (see Section~\ref{sec:PXC}) made the straightforward
approach used in \cite{ConDob,MC3} in the case of cubic graphs impossible in the case of valence $4$. This obstacle
was finally overcome in \cite{lost} and now a complete list of dart-transitive tetravalent graphs of order up $640$
is described in \cite{CubicCensus} and available in magma code at \cite{HATcensusSite}. 
 We use $\AT[n,i]$ for these graphs to indicate the  $i$-th graph of order $n$ in  this magma file.

\subsection{Semi-transitive graphs from universal groups}

Every semi-transitive tetravalent graph arises from an infinite 
$4$-valent tree $T_4$ and a group $G$ acting on $T_4$ semi-transitively  and having a finite stabiliser
by quotienting out a normal semiregular subgroup of $G$. All such groups $G$ were determined in
 \cite{MarNed}. This result in principle enables the same approach as used in the case of tetravalent dart-transitive graphs,
 and indeed, by overcoming the issue of semi-transitive graphs with large automorphism groups (see \cite{genlost}),
 a complete list of semi-transitive tetravalent graphs (and in particular $\frac{1}{2}$-arc-transitive graphs) was obtained in
 \cite{HATcensus} and is available at \cite{HATcensusSite}.  We include graphs from this census with the designation $\HT[n,i]$.

\vskip 1cm
We now begin to show the notation and details of each construction used for graphs in the Census.

\section{Wreaths, unworthy graphs}
\label{sc:Wreaths}

A general {\em Wreath graph}, denoted $\W(n, k)$,
has $n$ bunches of $k$ vertices each, arranged in a circle; every vertex of bunch $i$ is adjacent to every vertex in bunches  $i+1$ and $i-1$.   More precisely, its vertex set is $\ZZ_n\times\ZZ_k$; edges are all pairs of the form $\{(i,r), (i+1, s)\}$. The graph $\W(n,k)$ is regular of degree $2k$.  If $n = 4$, then $\W(n, k)$ is isomorphic to $K_{2k, 2k}$ and its symmetry group is the semidirect product of  $S_{2k}\times S_{2k}$ with $\ZZ_2$.  If $n \neq 4$, then its symmetry group is the semidirect product of $S_k^n$ with $D_n$; this group is often called the {\em wreath product} of $D_n$ over $S_k$.

Those of degree 4 are the graphs $\W(n, 2)$.  Here, for simplicity,  we can notate the vertex $(i,0)$ as $A_i$, and $(i,1)$ as $ B_i$ for $i \in \ZZ_n$, with edges $\{A_i, A_{i+1}\} , \{A_i, B_{i+1}\} , \{B_i, A_{i+1}\} , \{B_i, B_{i+1}\}$.  For example, Figure \ref{fig:W72} shows $W(7,2)$.

\begin{figure}[hhh]
\begin{center}
\epsfig{file=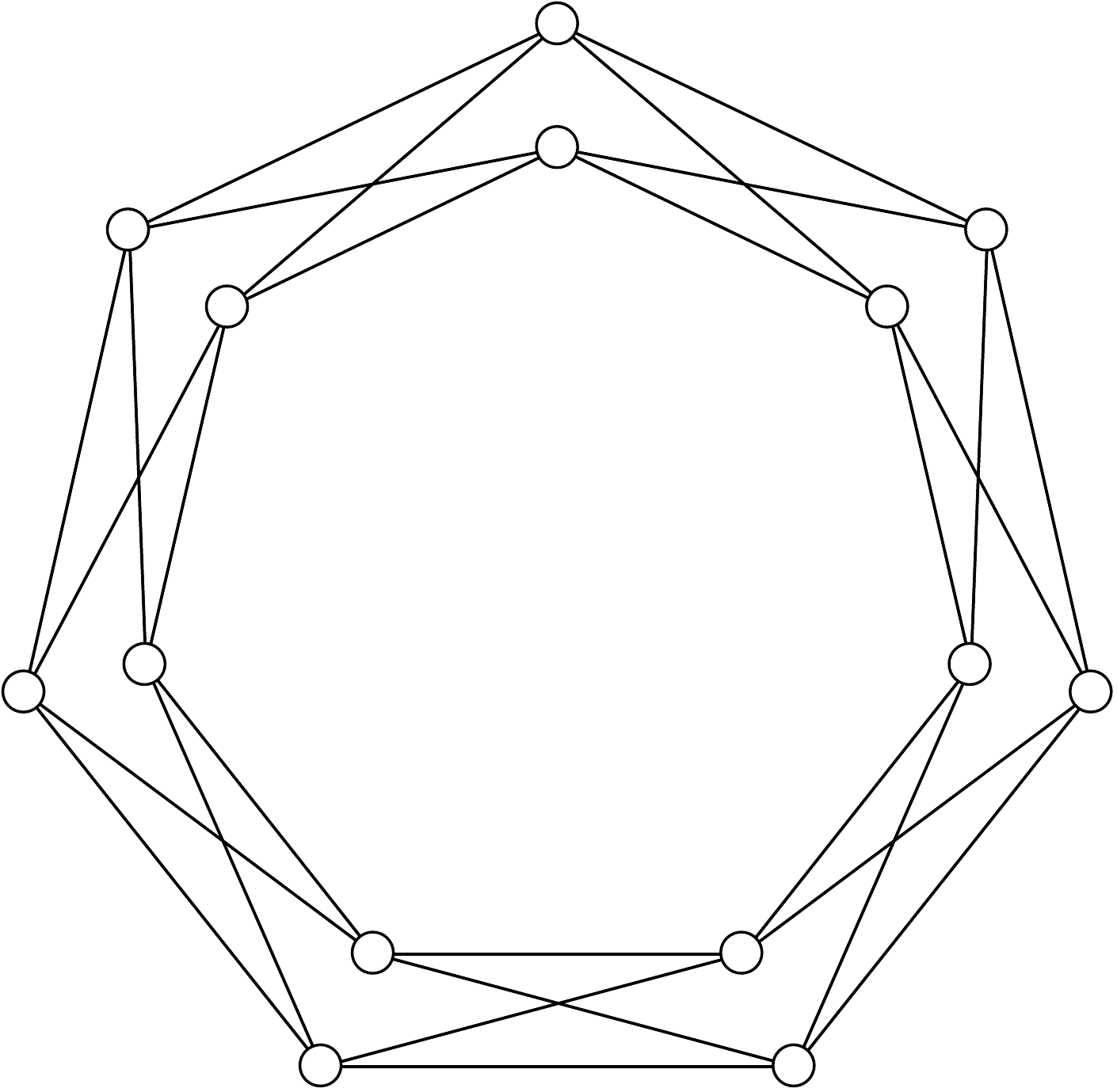,height=45mm}
\caption{$W(7,2)$}
\label{fig:W72}
\end{center}
\end{figure}

A wreath graph $\W(n, 2)$ has dihedral symmetries $\rho$ and $\mu$, where $\rho$ sends $(i,j)$ to $(i+1,j)$ and $\mu$ sends $(i, j)$ to $(-i,j)$.  The important aspect of this graph is that for each $i \in \ZZ_n$, there is a symmetry $\sigma_i$, called a ``local'' symmetry,  which interchanges $(i,0)$ with $(i, 1)$ and leaves every other vertex fixed.   Notice that $\rho$ acts as a shunt for a cycle of length $n$, and $\rho\sigma_0$ acts as a shunt for a cycle of length $2n$.  The third orbit of consistent cycles are those of the form $[A_i, A_{i+1}, B_i, B_{i+1}]$ 
(and their reverses).
The symmetry $\rho\mu\sigma_0$ is a shunt for one cycle in this orbit.

Since every $\sigma_i$ for $i \neq 0$ is in the stabilizer of $A_0$, 
we see that vertex stabilizers in these graphs can be arbitrarily large.

	A graph $\Gamma$ is {\em unworthy} provided that some two of its vertices have exactly the same neighbors.  The graph $W(n, 2)$ is unworthy because for each  $i$, the vertices $A_i$ and $B_i$ have the same neighbors.  The symmetry groups of vertex-transitive unworthy graphs tend to be large due to the symmetries that fix all but two vertices sharing the same neighborhood.

The paper \cite{g34} shows that there are only two kinds of tetravalent edge-transitive graphs which are unworthy.  One is the dart-transitive $\W(n,2)$ graphs.  The other is the ``sub-divided double" of a dart-transitive graph; this is a semisymmetric graph given by this construction:
	
\begin{construction}
 	Suppose that $\Lambda$ is  a tetravalent graph.  We construct a bipartite graph $\Gamma = \SDD(\Lambda)$ in the following way.  The white vertices of $\Gamma$ correspond to edges of $\Lambda$.  The black vertices correspond two-to-one to vertices of $\Lambda$; for each $v \in \V(\Lambda)$, there are two vertices $v_0, v_1$ in $\V(\Gamma)$.  An edge of $\Gamma$ joins each $e$ to each $v_i$ where $v$ is a vertex of $e$ in $\Lambda$.
 \end{construction}
 
 	The Folkman graph on 20 vertices \cite{F} is constructible as $\SDD(K_5)$.  It is clear that $\SDD(\Lambda)$ is tetravalent, and that if $\Lambda$ is dart-transitive, then  $\SDD(\Lambda)$ is edge-transitive.

	The paper \cite{g34} shows that any unworthy edge-transitive graph is isomorphic to some $\W(n, 2)$ if it is dart-transitive, and to
$\SDD(\Lambda)$ for some dart-transitive $\Lambda$ if it is semisymmetric.  There are no unworthy $\half$-transitive graphs.
	
\section{Circulants}
\label{sc:Circulants}

	In general, the {\em circulant} graph $C_n(S)$ is the Cayley graph for $\ZZ_n$  with generating set $S$.  Here $S$ must be a subset of $\ZZ_n$ which does not include 0, but does, for each $x \in S$, include $-x$ as well.  Explicitly,  vertices are $0, 1, 2, \dots, n-1$, considered as elements of  $\ZZ_n$, and two numbers $i, j$ are adjacent if their difference is in $S$.  Thus the edge set consists of all pairs $\{i, i+s\}$ for $i \in \ZZ_n$, and $s \in S$.   If $S = \{\pm a_1, \pm a_2, \dots\}$, we usually abbreviate the  name of each graph as $C_n(a_1, a_2, \dots)$. The numbers $a_i$ are called {\em jumps} and the set of edges of the form  $\{j, j+a_i\}$ for a fixed $i$ is called a {\em jumpset}.
Figure \ref{fig:C1013} shows an example, the graph $C_{10}(1, 3)$.

\begin{figure}[hhh]
\begin{center}
\epsfig{file=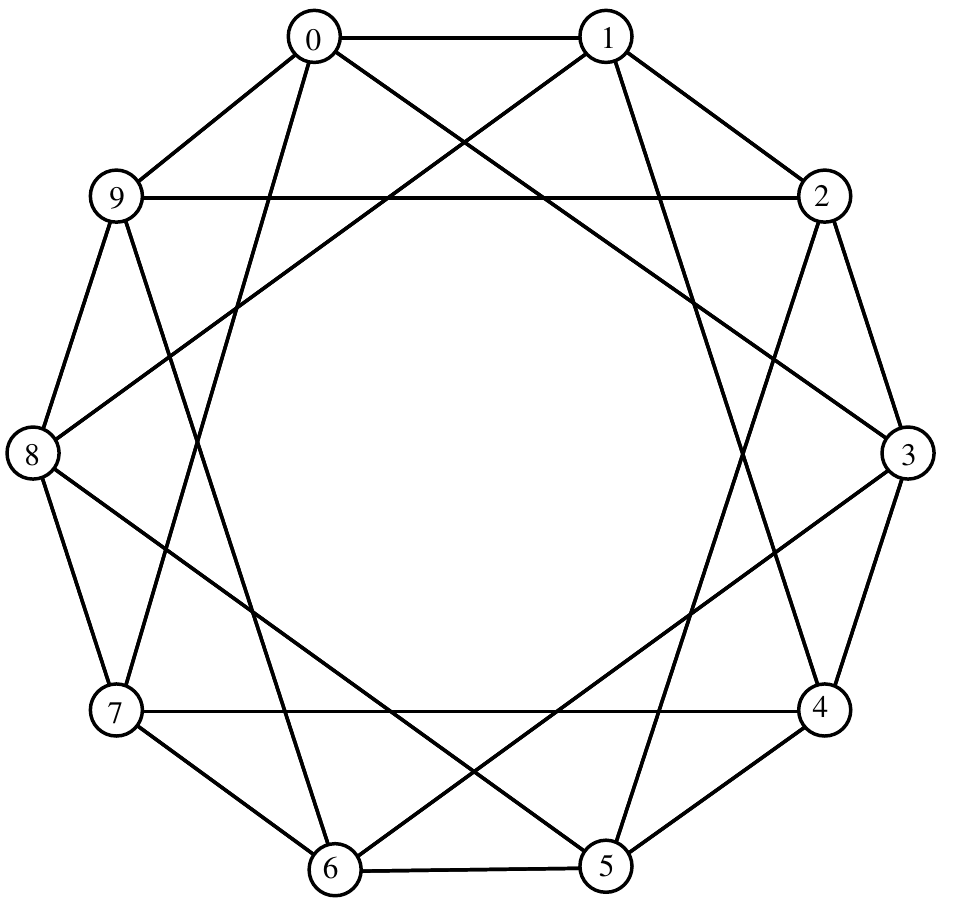,height=45mm}
\caption{$C_{10}(1,3)$}
\label{fig:C1013}
\end{center}
\end{figure}

	For general $n$, if $S$ is a subgroup of the group $\ZZ_n^*$ of units mod $n$, 
then $C_n(S)$ is dart-transitive, though it may be so in many other circumstances as well.  
In the tetravalent case, there are two possibilities:
	
\begin{theorem}
\label{th:Circ}
If $\Gamma$ is a tetravalent edge-transitive circulant graph with $n$ vertices, then
it is dart-transitive and either:
\begin{itemize}
\item[{\rm (1)}] $\Gamma$ is isomorphic to $C_n(1, a)$ for some $a$ such that $a^2 \equiv \pm 1 (\mod n)$, or
\item[{\rm (2)}]
$n$ is even, $n = 2m$, and $\Gamma$ is isomorphic to $C_{2m}(1, m+1)$.
\end{itemize}
\end{theorem}

\begin{proof}
Note first that every edge-transitive circulant $C_n(a,b)$ is dart-transitive, due
to an automorphism which maps a vertex $i$ to the vertex $a-i$ mod $n$, and thus inverts the
edge $\{0,a\}$.

In (1), the dihedral group $D_n$ acts transitively on darts of each of the two jumpsets, and if $a^2 \equiv \pm 1\> (\mod n)$, then the multiplication by $a$ mod $n$ induces a symmetry of the graph which interchanges the two jumpsets.

On the other hand,  in (2), $C_{2m}(1, m+1)$ is isomorphic to the unworthy graph $\W(m, 2)$, with $i$ and $i+m$ playing the roles of $A_i$ and $B_i$.  Thus, the sufficiency of  (1) or (2) for edge-transitivity is clear.

The necessity can be deduced either from a complete classification of dart-transitive circulants of arbitrary valence proved in \cite{CircAT1} and \cite{CircAT2} or by careful examination of short cycles in dart-transitive circulants.
\end{proof}

\section{Toroidal Graphs}
\label{sc:Toroidal Graphs}

The tessellation of the plane into squares is known by its Schl\"afli symbol, $\{4, 4\}$.  Let $T$ be the group of translations
of the plane that preserve the tessellation. Then $T$ is isomorphic to $\ZZ\times \ZZ$ and acts transitively on the vertices of the
tessellation.  If $U$ is a subgroup of finite index in  $T$, then $\M = \{4, 4\}/U$ is a finite map  of type $\{4, 4\}$ on the torus, and every such map arises in this way.   A symmetry $\alpha$ of  $\{4, 4\}$ acts as a symmetry of $\M$ if and only if $\alpha$ normalizes $U$.  Thus every such $\M$ has its symmetry group $\Aut(\M)$ transitive on vertices, on horizontal edges, on vertical edges;  further for each edge $e$ of $\M$, there is a symmetry reversing the edge (and acting as a $180^\circ$ rotation about its center).  Thus, $\Aut(\M)$ is transitive on the edges of $\M$ (and so must be dart-transitive) if and only if $U$ is normalized by some $90^\circ$ rotation or by some reflection about some axis at a $45^\circ$ angle to the axes.  As shown in \cite{STW}, this can happen in three different ways:

\begin{itemize}
\item[{\rm (1)}]
$\{4, 4\}_{b, c}$:   For this graph and map, $U$ is the group generated by the translations $(b, c)$ and $(-c, b)$.   These are the well-known {\em rotary} maps.  $\{4, 4\}_{b, c}$ has $D = b^2+c^2$ vertices, $D$ faces and $2D$ edges.  Figure \ref{fig:4432} shows the case when $b=3, c=2$.

\begin{figure}[hhh]
\begin{center}
\epsfig{file=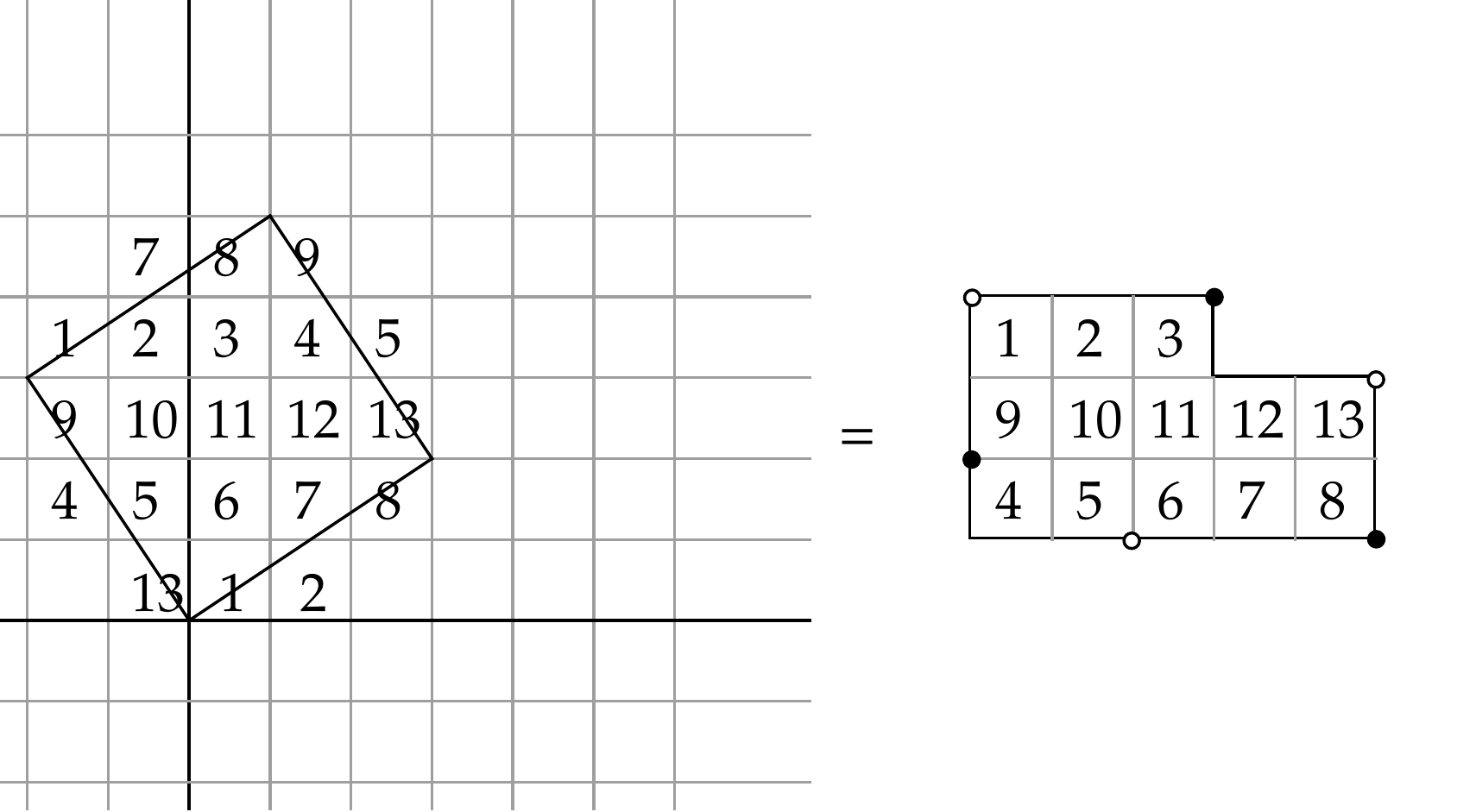,height=40mm}
\caption{The map $\{4, 4\}_{3,2}$}
\label{fig:4432}
\end{center}
\end{figure}

\item[{\rm (2)}]	$\{4, 4\}_{<b, c>}$:   This graph and map, defined for $ b-1 > c \ge 0$, uses for $U$ the group generated by the translations $(b, c)$ and $(c, b)$.  It has $E = b^2-c^2$ vertices, $E$ faces and $2E$ edges.  Figure \ref{fig:44<31>} shows the  map $\{4, 4\}_{<3,1>}$.

\begin{figure}[hhh]
\begin{center}
\epsfig{file=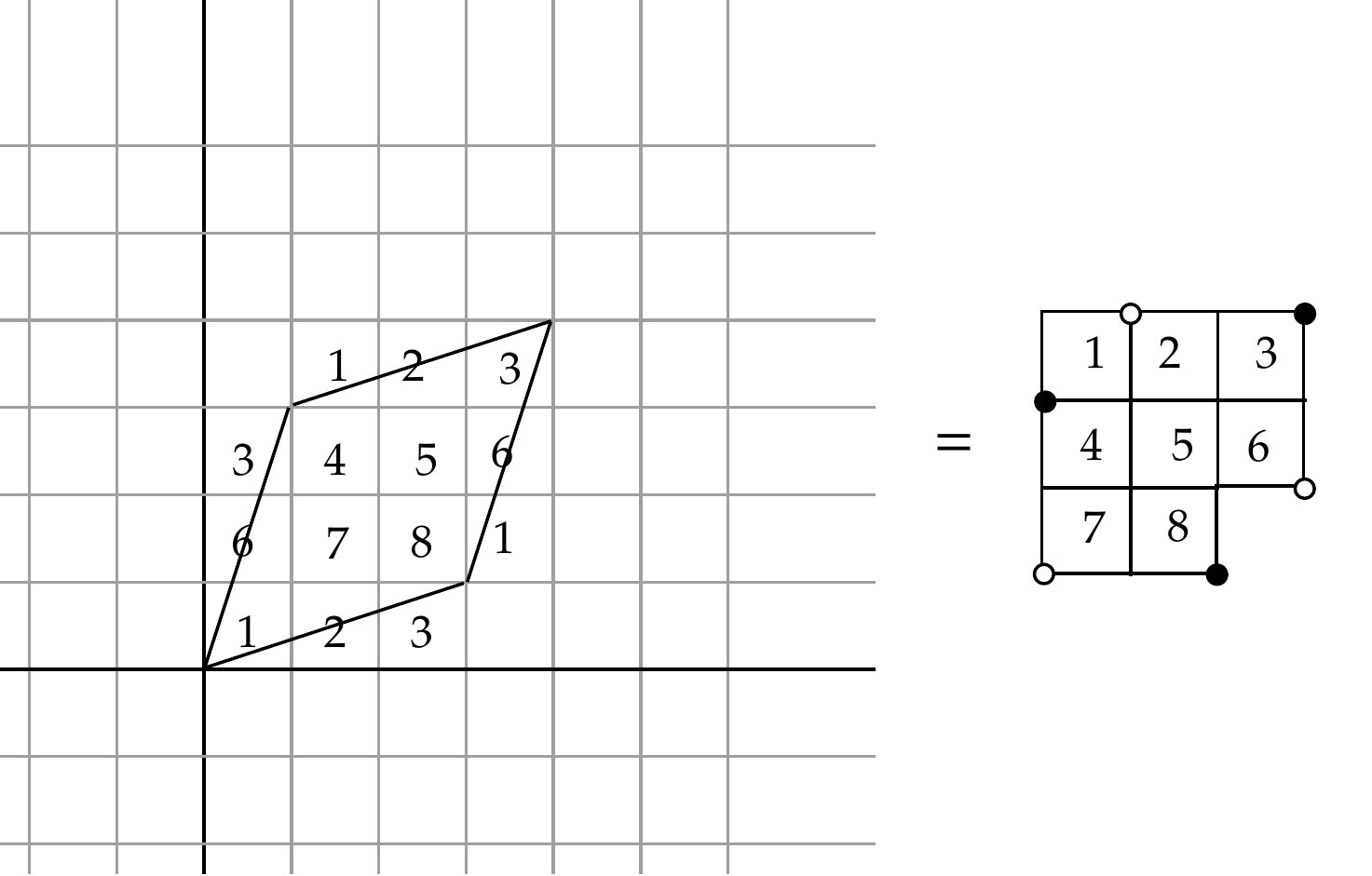,height=40mm}
\caption{The map $\{4, 4\}_{<3,1>}$}
\label{fig:44<31>}
\end{center}
\end{figure}

	Notice that we exclude the case $b = c+1$.  The {\em map} $\{4, 4\}_{<c+1, c>}$ exists, but its underlying multigraph has parallel edges.

\item[{\rm (3)}]	$\{4, 4\}_{[b, c]}$:   For this graph and map, $U$ is the group generated by the translations $(b, b)$ and $(-c, c)$.  It is defined only for $b \ge c > 1$. It has $F =2bc$ vertices, $F$ faces and $2F$ edges.

\begin{figure}[hhh]
\begin{center}
\epsfig{file=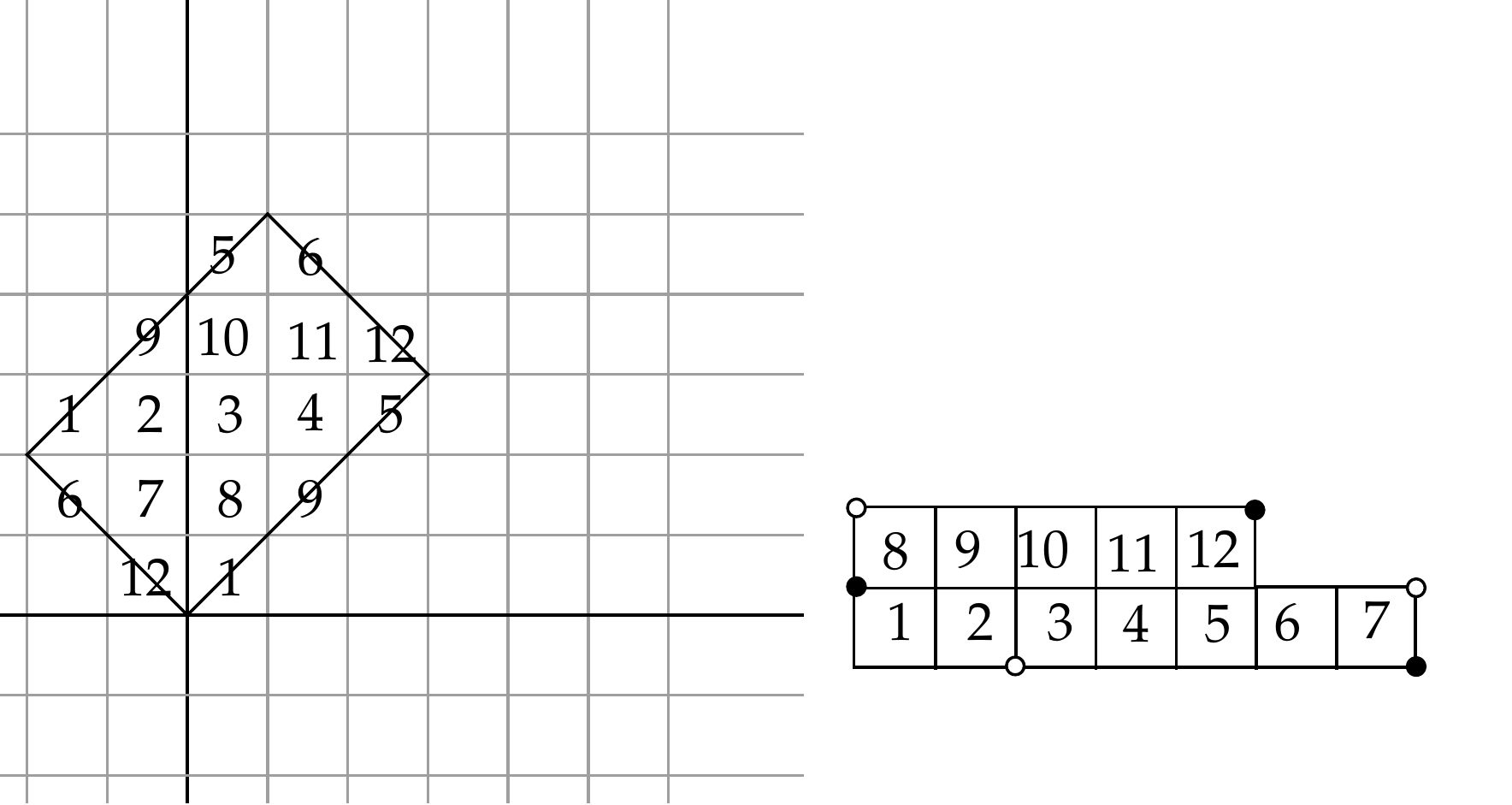,height=35mm}
\caption{The map $\{4, 4\}_{[3,2]}$}
\label{fig:44[32]}
\end{center}
\end{figure}

If $c = 1$, then the map $\{4, 4\}_{[b, c]}$ exists, but, again,  its skeleton has multiple (parallel) edges and so is not a simple graph.
\end{itemize}

Because $\{4, 4\}_{b, 0}$ is isomorphic to $\{4, 4\}_{<b, 0>}$, this map is reflexible.  Because $\{4, 4\}_{b, b}$ is isomorphic to $\{4, 4\}_{[b, b]}$, this map is also reflexible.  All other  $\{4, 4\}_{b, c}$ are  {\em chiral}:  i.e., rotary but not reflexible.

	We will use the symbols $\{4, 4\}_{b, c}$, $\{4, 4\}_{[b, c]}$, $\{4, 4\}_{<b, c>}$ to stand for the graphs which are skeletons of these maps as well as for the maps themselves.  It is surprising that  $\{4, 4\}_{b+c, b-c}$ is a double cover of $\{4, 4\}_{b, c}$, while $\{4, 4\}_{[b+c, b-c]}$ is a double cover of $\{4, 4\}_{<b, c>}$ and    $\{4, 4\}_{<b+c, b-c>}$ is  a double cover of $\{4, 4\}_{[b, c]}$.

Now, every $U$ of finite index can be expressed in the form $U = \langle(d, e), (f, g)\rangle$, where $(d, e)$ and $(f, g)$ are linearly independent.  In particular, it  is not hard to show that $U$ can also be expressed in the form $U = \langle(r, 0), (s, t)\rangle$, where $t = {\rm GCD}(e, g)$.  Let $e = e't, g = g't$, and let $m$ and $n$ be Bezout multipliers, so that $ me'+ng' =1$.  Then $s = md + nf$ and $r = g'd - e'f$.  This gives a fundamental region which is a rectangle $r$ squares wide, $t$ squares high, with the left and right edges identified directly, and the bottom edges identified with the top after a shift $s$ squares to the right, as in Figure \ref{fig:44form}.

\begin{figure}[hhh]
\begin{center}
\epsfig{file=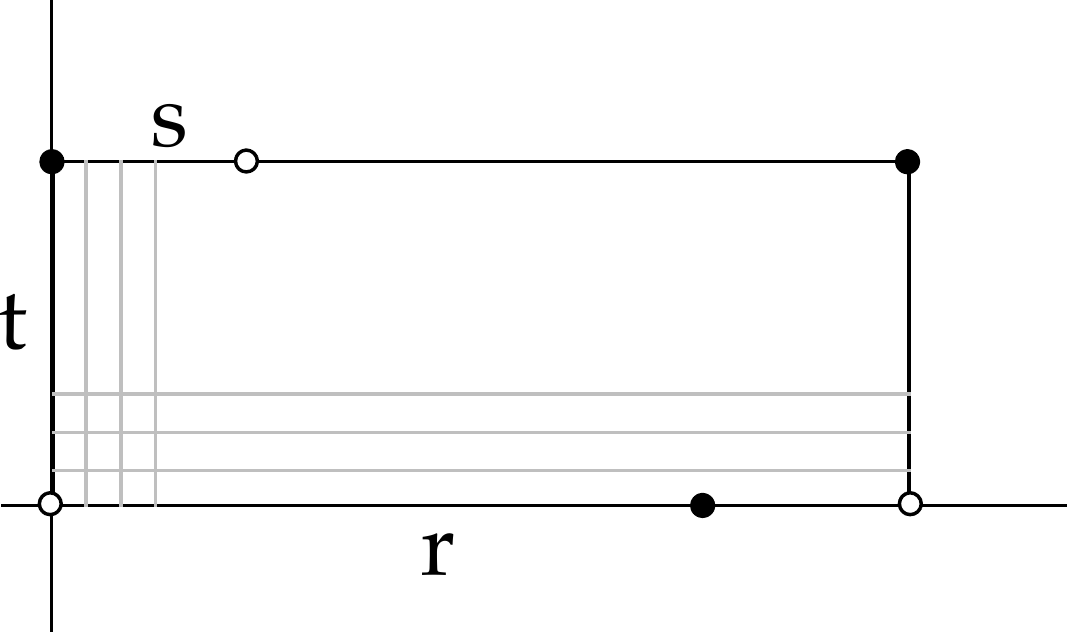,height=35mm}
\caption{A standard form for maps of type $\{4, 4\}$}
\label{fig:44form}
\end{center}
\end{figure}

In the special case in which $b$ and $c$ are relatively prime, we have $t = 1$, and this forces the graph to be circulant; in fact, the circulant graph $C_r(1, s)$.   Here, $s^2 \equiv -1$ for  $\{4, 4\}_{b, c}$, and $s^2 \equiv 1$ for both   $\{4, 4\}_{<b, c>}$ and    $\{4, 4\}_{[b, c]}$.  On the other hand, every circulant graph is toroidal or has an embedding on the Klein bottle.  More precisely, if $a^2\equiv -1$(mod $n$), then  $C_n(1,a) \cong \{4, 4\}_{b, c}$ for some $b, c$.  If $a^2\equiv 1$(mod $n$), then $C_n(1,a) \cong \{4, 4\}_{<b, c>}$ of $\{4, 4\}_{[b, c]}$for some $b, c$.
Of the graphs $C_{2m}(1, m+1)\cong \W(m, 2)$, if $m$ is even then then it is $\{4, 4\}_{[\frac{m}{2}, 2]}$.  On the other hand,  if $m$ is odd, it has an embedding on the Klein bottle, though that embedding is not edge-transitive \cite{maps}.

\section{Depleted Wreaths}
\label{sc:Depleted Wreaths}

The {\em general Depleted Wreath graph} $\DW(n, k)$ is formed from $\W(n, k)$ by removing $k$ disjoint $n$-cycles, each of these cycles containing one vertex from each of the $k$ bunches.  More precisely, its vertex set is $\ZZ_n\times\ZZ_k$.   Its edge set is the set of all pairs of the form $\{(i,r), (i+1, s)\}$ for $i \in \ZZ_n$ and $r, s \in \ZZ_k$, $r\ne s$.  Its vertices are of degree $2(k-1)$.  It is tetravalent when $k = 3$. Figure \ref{fig:DW} shows part of $\DW(n, 3)$.

\begin{figure}[hhh]
\begin{center}
\epsfig{file=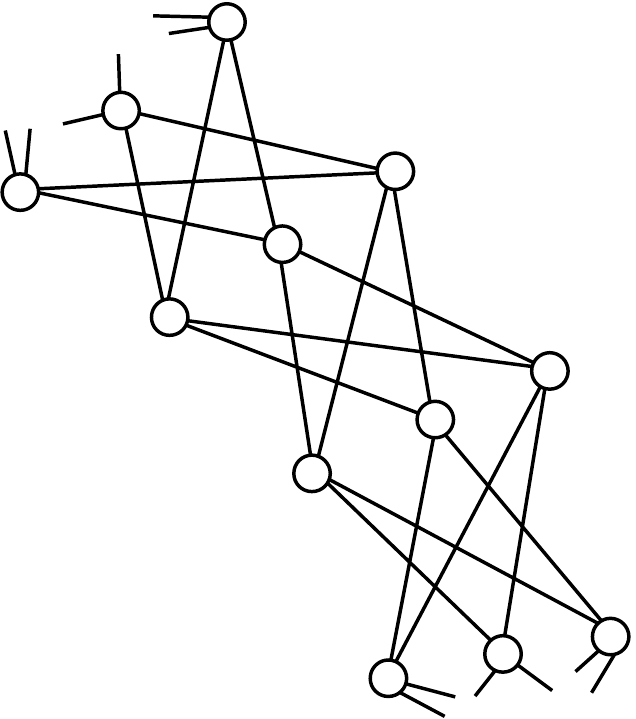,height=35mm}
\caption{Part of $\DW(n, 3)$}
\label{fig:DW}
\end{center}
\end{figure}

If $n \equiv 1$ (mod 3) then $\DW(n, 3) \cong C_{3n}(1, n+1)$ and if $n \equiv 2$ (mod 3) then $\DW(n, 3) \cong C_{3n}(1, n-1)$; in the remaining case, $n \equiv 0$ (mod 3), $\DW(n, 3)$ is not a circulant.

The Depleted Wreaths are also toroidal.  If $n$ is even, then $\DW(n, 3)\cong \{4,4\}_{[\frac{n}{2}, 3]}$, while if $n$ is odd, then  $\DW(n, 3)\cong \{4,4\}_{<\frac{n+3}{2}, \frac{n-3}{2}>}$

\section{Spidergraphs}
\label{sc:Spidergraphs}

	The {\em Power Spidergraph} $\PS(k, n; r)$ is defined for $k\ge 3, n \ge 5,$ and $r$ such that $r^k \equiv \pm 1$ (mod $n$), but  $r \not \equiv \pm 1$ (mod $n$).  Its vertex set is $\ZZ_k \times\ZZ_n$, and vertex $(i,j)$ is connected by edges to vertices $(i+1, j \pm r^i)$.  It may happen that this graph is not connected; if so, re-assign the name to the connected component containing $(0, 0)$.  The resulting graph is always semitransitive.  Directing each edge from  $(i,j)$  to $(i+1, j \pm r^i)$ gives a semi-transitive orientation.
	
		Closely related is the  {\em Mutant Power Spidergraph} $\MPS(k, n; r)$.  It is defined for $k\ge 3$, $n$ even and $n \ge 8,$ and $r$ such that $r^k \equiv \pm 1$ (mod $n$), but  $r \not \equiv \pm 1$ (mod $n$).  Its vertex set is $\ZZ_k \times\ZZ_n$.  For $0\le i < k-1$, vertex $(i,j)$ is connected by edges to vertices $(i+1, j \pm r^i)$; vertex $(k-1, j)$ is connected to $(0, j \pm  r^{k-1} + n/2)$.  Maru\v si\v c  \cite{M1} and \v Sparl \cite{S1} have shown that every tetravalent tightly-attached graph is isomorphic to some $\PS$ or $\MPS$ graph, and that the graph is $\frac{1}{2}$-arc-transitive in all but a few cases: if $r^2 \equiv \pm 1$ (mod $n$), then the graph is dart-transitive.  The very special graph  $\Sigma = \PS(3, 7; 2)$  is dart-transitive.   If $m$ is an integer not divisible by 7 and $r$ is the unique  solution mod $n = 7m$ to $r \equiv 5$ (mod 7), $r \equiv 1$ (mod $m$), then $\PS(6, n; r)$ (which is a covering of $\Sigma$) is dart-transitive.

		 The paper  \cite{S1} defines and notates these graphs in ways which differ from this paper, and the difference is worthy of note.  If $m, n, r, t$ are integers satisfying (1) $m, n $  are even and at least 4, (2)  $r^m \equiv 1$ (mod $n$) and (3)  $s = 1 + r + r^2 + \dots r^{m-1} +2t$ is equivalent to 0 (mod $n$), then \cite{S1} defines the graph $X_e(m,n;r,t)$ (``e" stands for ``even") to have vertices  $[i,j]$ with $i \in \ZZ_m$ and $j \in \ZZ_n$ and edges from $[i, j]$ to $[i+1,j]$ and $[i+1,j+r^i]$ when $0 \le i <m-1$, while $[m-1, j]$ is connected to $[0, j+t]$ and $[0, j+r^{m-1} + t]$. 
		 
	The argument in \cite{STG} shows that if  (1), (2) and (3) hold, then $r^m \equiv 1$ (mod $2n$).    Then it is not hard to see that if  the integer $s$ is equivalent to 0 mod $2n$, then  $X_e(m,n;r,t)$ is isomorphic to $\PS(m, 2n;r)$; if $s$ is equivalent to $n$ mod $2n$, then it is isomorphic to $\MPS(m, 2n;r)$.

\section{Attebery Graphs}
\label{sc:Attebery}

The following quite general construction is due to  Casey Attebery \cite{CAthesis}.   We define the graph 
$\Att(A, T, k;a, b)$  to be underlying graph of the digraph $\Att[A, T, k;a, b]$.   The parameters are: 
an abelian group $A$, an automorphism 
$T$ of $A$, an integer $k>2$,  and two elements $a$ and $b$ of $A$.  
Let $c = b-a$ and define $a_i, b_i, c_i$ to 
be $aT^i, bT^i, cT^i$ respectively for $i = 0, 1, 2, \dots, k$.  We require that:
\begin{itemize}
\item[{\rm (1)}]
 $\{a_k, b_k\} = \{a, b\}$, 
 \item[{\rm (2)}]
 $A$ is  generated by $c_0, c_1, c_2, \dots, c_{k-1}$ and $\Sigma_{i = 0}^{k-1}a_i$; and 
 \item[{\rm (3)}]
  $a+b$ is in the kernel of the endomorphism $T^* = \Sigma_{i = 0}^{k-1}T^i$.
 \end{itemize} 
Then the vertex set  of the digraph $\Att[A, T, k;a, b]$ and of the graph $\Att(A, T, k;a, b)$ is defined to be 
$A\times\ZZ_k$.  In the digraph, edges lead 
from each $(x, i)$ to $(x+a_i, i+1)$ and $(x+b_i, i+1)$.  Then the digraph is a semi-transitive orientation 
of the graph.  The graph is thus semi-transitive, and it is often, but not always, $\frac{1}{2}-$arc transitive.

Not all Attebery graphs are implemented in the Census.  There are four special cases which are:
\begin{enumerate}
\item   If $A=\ZZ_n$ and $T$ is multiplication by $r$, the Attebery graph is just $\PS(k, n; r)$, and thus the Attebery graphs are generalizations of the spidergraphs.
\item  The graph called $C^{\pm 1}(p;st,s)$ in \cite{GP}.  This is an Attebery graph with $A = \ZZ_p^s, 
k = st, T:(a_1, a_2, \dots, a_s) \to (a_2, a_3, \dots, a_s, a_1), -b = a = (1, 0, 0 , \dots, 0)$.
\item The graph called $C^{\pm e}(p;2st,s)$ in \cite{GP} with $e^2 \equiv -1$ (mod p).  This is an Attebery graph with $A = \ZZ_p^s, k = 2st, T:(a_1, a_2, \dots, a_s) \mapsto (ea_2, ea_3, \dots, ea_s, ea_1), -b = a = (1, 0, 0 , \dots, 0)$.
\item   $\AMC(k,n,M)$ is $\Att(A, T, k;a, b)$ where $A$ is $\ZZ_n\times\ZZ_n$, $M$ is a $2\times 2$ matrix over $\ZZ_n$ satisfying $M^k = \pm I$,  $T$ is multiplication by $M$, and $a = (1, 0), b = (-1,0)$.
\end{enumerate}

The second and third of these are  generalized to the graph $\CPM(n, s, t, r)$, defined later in this paper.

 It is intriguing that even though the matrix
 $$ M=
  \left[
         \begin{array}{c c}
          1    &  -4  \\
	  4 &  1
        \end{array}
       \right]
  $$     
 does not satisfy the condition $M^4 = \pm I$, the graph
 $\AMC(4, 12, M)$
 is nevertheless edge-transitive, and in fact semisymmetric.
 Even more striking is that so far we have no other construction for this graph.

\section{The separated Box Product}
\label{SepBox}

	Suppose that $\Delta_1$ and $\Delta_2$ are digraphs in which every vertex has in- and out-valence 2.  We allow $\Delta_1$ and $\Delta_2$ to be non-simple.

 We form the {\em separated box product} $\Delta_1\#\Delta_2$ as the underlying graph of the orientation whose vertex set is $\V(\Delta_1)\times\V(\Delta_2)\times\ZZ_2$, and whose edge set contains two types of edges:  ``horizontal'' edges join $(a, x, 0)\rightarrow (b, x, 1)$, and ``vertical'' edges $(b, x, 1)\rightarrow (b, y,0)$, where $a\rightarrow b$ in $\Delta_1$ and $x\rightarrow y$ in $\Delta_2$.   

	An orientation is {\em reversible} provided that it is isomorphic to its reversal.   There are several useful cases of this construction:
 
\begin{itemize}
\item When $\Delta_1 = \Delta_2$ and $\Delta_1$ is reversible,  then $\Delta_1\#\Delta_2$ is dart-transitive. 
\item When $\Delta_1 = \Delta_2$ and $\Delta_1$ is not reversible,  $\Delta_1\#\Delta_2$ has  a semitransitive orientation and so might be dart-transitive or $\frac{1}{2}$-transitive.
\item When $\Delta_1$ is not isomorphic to  $\Delta_2$ or its reverse but both are reversible,  then $\Delta_1\#\Delta_2$ has an LR structure.
\item  When $\Delta_1$ is not isomorphic to  $\Delta_2$ but is isomorphic to the reverse of $\Delta_2$,  then $\Delta_1\#\Delta_2$ is at least bi-transitive  and is semisymmetric in all known cases.   
\end{itemize}
In the Census, we use for $\Delta_1$ and $\Delta_2$  directed graphs from the census of  $2$-valent dart-transitive digraphs
\cite{HATcensus,HATcensusSite} with notation ${\rm ATD}[n,i]$ for the $i^{th}$ digraph of order $n$ from that census.  We also allow the ``sausage digraph''
 ${\rm DCyc}_n$: an $n$-cycle with each edge replaced by two directed edges, one in each direction.
See \cite{SepBox} for more details.

\section{Rose Windows}
\label{sc:Rose Windows}

The {\em Rose Window graph} $\RW_n(a,r)$  has $2n$ vertices: $A_i, B_i$ for $i \in \ZZ_n$. The graph has four kinds of edges:

\begin{tabular}{rl}
Rim: &$A_i  -  A_{i+1}$\\
In-Spoke:& $A_i  -  B_i$\\
Out-spoke: &$B_i  -  A_{i+a}$\\
Hub:& $B_i  -  B_{i+r}$
\end{tabular}

For example, Figure~\ref{Fi:R1225} shows $\RW_{12}(2, 5)$.

\begin{figure}[h!]
\begin{center}
\epsfig{file=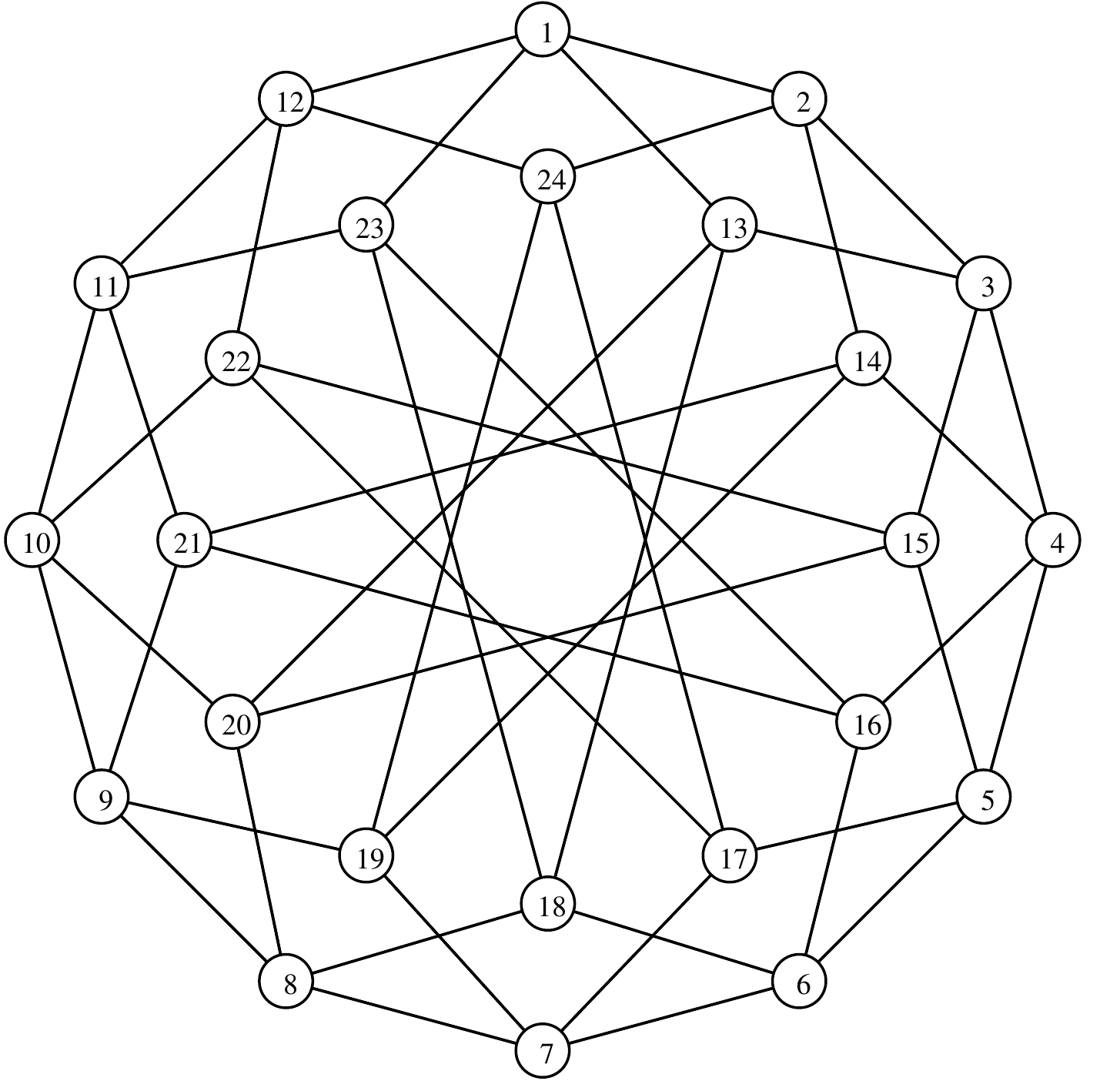, height = 50mm}
\caption{$\RW_{12}(2, 5)$}
\label{Fi:R1225}
\end{center}
\end{figure}

We will soon mention these graphs as examples of bicirculant graphs, and more generally later as examples of polycirculant graphs. 

The paper \cite{KKM} shows that every edge-transitive $\RW_n(a,r)$ is dart-transitive and is isomorphic to one of these:
\begin{enumerate}
\item[a.] $\RW_n(2, 1)$.     (This graph is isomorphic to $\W(n, 2)$.)
\item[b.] $\RW_{2m}(m+2, m+1)$.
\item[c.] $\RW_{12k}(3k\pm 2, 3k\mp 1)$.
\item[d.] $\RW_{2m}(2b, r)$, where $b^2\equiv \pm 1$ (mod $m$), $r$ is odd and $r \equiv 1$ (mod $m$).
\end{enumerate}

\section{Bicirculants}
\label{sc:Bicirculants}

A graph $\Gamma$ is {\em bicirculant} provided that it has a symmetry $\rho$ which acts on $\V$ as two cycles of the same length.  A rose window graph is bicirculant, as the symmetry sending $A_i$ to $A_{i+1}$ and  $B_i$ to $B_{i+1}$ is the required $\rho$.

	The paper \cite{KKMW} classifies all edge-transitive tetravalent bicirculant graphs.  Besides the Rose Window graphs there is one other class of graphs, called ${\rm BC4}$ in that paper, and called $\BC$ here and in the Census.  The graph $\BC_n(a, b, c, d)$ has $2n$ vertices: $A_i, B_i$ for $i \in \ZZ_n$.  The edges are all pairs of the form $\{A_i, B_{i+e}\}$ for   $i \in \ZZ_n, e \in \{a, b, c, d\}$.  It is easy to see that any such graph is isomorphic to one of the form $\BC_n(0, a, b, c)$ where $a$ divides $n$.  The edge-transitive graphs in this class consists of three sporadic examples and three infinite families.
The sporadics are:
$$ \BC_7(0, 1, 2, 4),\>
\BC_{13}(0, 1, 3, 9),\
\BC_{14}(0, 1, 4, 6)$$
	
	Of the three infinite families of graphs   $\BC_n(0, a, b, c)$ , there are two in which we can choose $a = 1$, and a third, less easy to describe, in which none of the parameters is relatively prime to $n$:
	\begin{enumerate}
\setlength{\itemsep}{0pt}	
	\item [(I)]  $\BC_n(0, 1, m+1, m^2+m+1)$, where $(m+1)(m^2+1) = 0$ (mod $n$)
	\item [(II)] $\BC_n(0, 1, d, 1-d)$, where $2d(1-d) = 0$ (mod $n$)
	\item [(III)] $\BC_{krst}(0, r,  rs's + st,  rt't+st+rst)$, where
\begin{itemize}
\setlength{\itemsep}{0pt}	
\item[{\rm (1)}] $r, s, t$ are all integers greater than $1$,
\item[{\rm (2)}] $s, t$ are odd,
\item[{\rm (3)}] $r, s, t $ are relatively prime in pairs,
\item[{\rm (4)}] $k \in \{1, 2\}$,
\item[{\rm (5)}] if $k = 2$, then $r$ is even,
\item[{\rm (6)}] $s'$ is an inverse of $s$ mod $krt$, and $t'$ is an inverse of $t$ mod $krs$.  
\end{itemize}
\end{enumerate}

\section{Semiregular symmetries and their diagrams}
\label{sc:Diagrams}
	A symmetry $\sigma$ is {\em semiregular}  provided that it acts on the $|\V|=kn$ vertices as $k$ cycles of length $n$.  We can visually represent such a graph and symmetry with a {\em diagram}.  This is a graph-like object, with labels.  Each ``node'' represents one orbit under $\sigma$.  
If $\sigma = (u_0, u_1, \dots, u_{n-1})(v_0, v_1, \dots, v_{n-1})\ldots(w_0, w_1, \dots, w_{n-1})$
 and there is an edge from $u_0$ to $v_{a}$, then there is an edge from each $u_i$ to $v_{a+i}$ (indices computed modulo $n$).
This matching between $u_i's$ and $v_i's$ is represented in the diagram  by a directed edge from node $u$ to node $v$ with label $a$ (or one from $v$ to $u$ with label $-a$). If $a=0$, then the label $a$ and the direction of the edge in the diagram can be omitted.

If  there is an edge from $u_0$ to $u_{b}$ (and thus one from each $u_i$ to $u_{i+b}$),
we represent this by a loop at $u$ with label $b$.  In the special case in which $n$ is even and $b = \frac{n}{2}$, there are only $\frac{n}{2}$ edges in the orbit and we represent them with a semi-edge at $u$.  This convention makes the valence of $u$ in the diagram the same as the valences of all of the $u_i$'s.  Finally, there is one label ``mod $n$'' on the entire diagram.

It should be pointed out that what we were describing above is  simply  the notion of a quotient of a graph (as defined, for example, in \cite{MNS}) by   the cyclic group generated by the semiregular element $\sigma$, and the diagram which we obtain is the voltage graph describing the graph $\Gamma$.

  As an example, consider the bicirculant graphs  $\RW_n(a, r)$ and $\BC_n(0, a, b, c)$, which can be represented by the diagrams shown in Figure \ref{fig:RBC}.

\begin{figure}[hhh]
\begin{center}
\epsfig{file=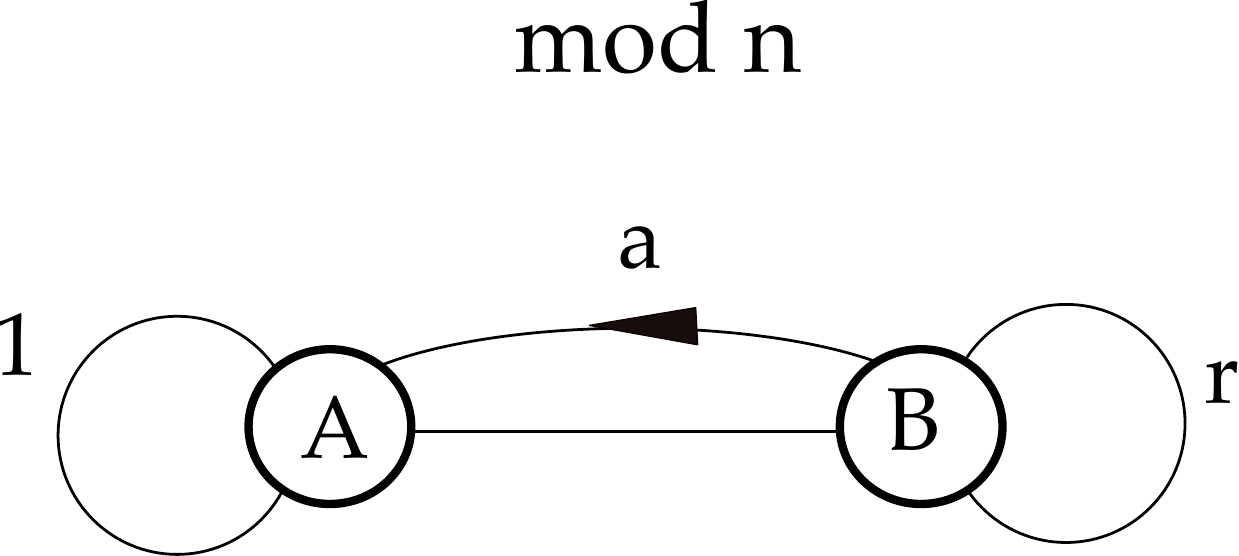,height=30mm}\hspace{32pt}\epsfig{file=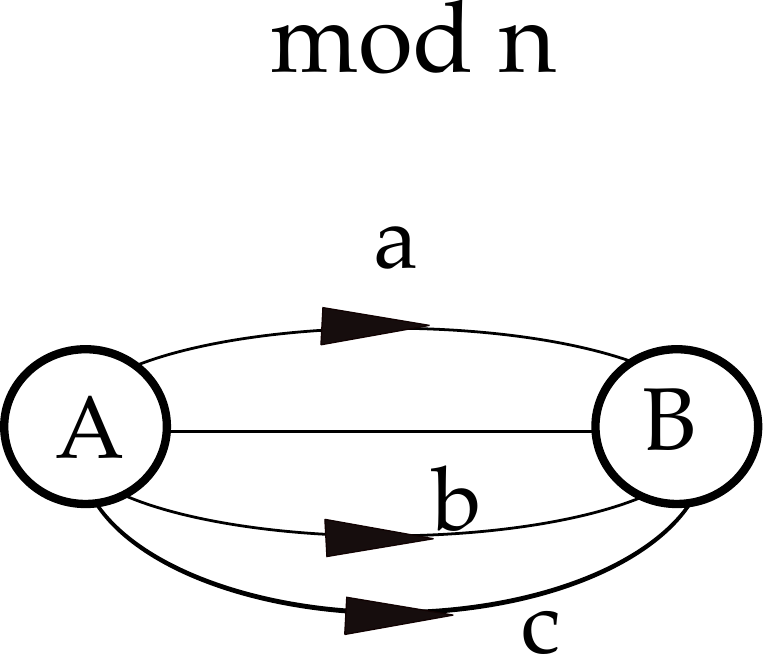,height=35mm}
\caption{Diagrams of graphs}
\label{fig:RBC}
\end{center}
\end{figure}

It is a simple matter to recover the graph from a diagram, and we use diagrams to define many families of graphs.

\subsection{Propellors}
\label{sc:Propellors}

	\begin{figure}[hhh]
\begin{center}
\epsfig{file=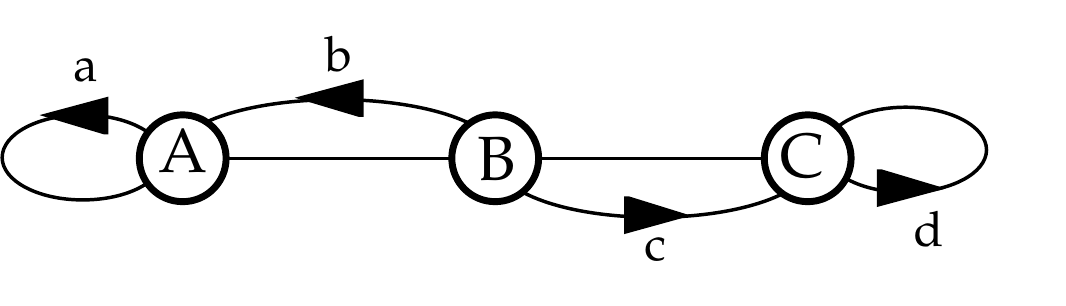,height=25mm}
\caption{Diagram for the graph $\Prr_n(a, b, c, d)$}
\label{fig:PropDiag}
\end{center}
\end{figure}

	A {\em Propellor Graph} is a graph with the diagram shown in Figure \ref{fig:PropDiag}.  This means that the graph has $3n$ vertices:$A_i, B_i, C_i$ for $i \in \ZZ_n$.   There are 6 kinds of edges:

\begin{tabular}{rl}
Tip: &$A_i  -  A_{i+a}$\\
&$C_i - C_{i+d}$\\
Flat:& $A_i  -  B_i$\\
&$B_i  -  C_i$\\
Blade:& $B_i  -  A_{i+b}$\\
& $B_i  -  C_{i+c}$
\end{tabular}

Propellor graphs have been investigated by Matthew Sterns.  He conjectures in  \cite{MSthesis} that the edge-transitive propellor graphs are isomorphic to

\begin{enumerate}
\item [I] $\Prr_n(1, 2d, 2, d)$, where $d^2 \equiv \pm 1$ (mod $n$)
\item [II] $\Prr_n(1, b, b+4, 2b+3)$, where $8b+16 \equiv 0$ (mod $n$) and $b \equiv 1$ (mod $4$)
\item [III] These five sporadic examples: 
$$\Prr_{5}(1, 1, 2, 2),   \Prr_{10}(1,1,2,2), \> \Prr_{10}(1, 4, 3, 2), \>\Prr_{10}(1, 1, 3, 3), \> \Prr_{10}(2, 3, 1, 4).$$
\end{enumerate}

Notice first that in every case except the last of the sporadic cases, the A-tip edges form a single cycle.  The first of the infinite families consists  of the {\em 2-weaving} graphs: those graphs in which some symmetry of the graphs sends this cycle to a cycle of the form AB AB AB$\ldots$  The second class consists of {\em 4-weaving} graphs: those in which some symmetry sends the A-tip cycle to a cycle of the form ABCB ABCB AB$\ldots$   There is also a class of {\em 5-weaving} graphs in which some image of the tip cycle is of the form AABCB AABCB AAB$\ldots$, but the requirements for this class force $n$ to divide 10, resulting in the first four of the sporadic cases.   This leaves the last sporadic graph as something of a mystery.

A more recent paper \cite{Pr} proves the conjecture, though it presents the graphs in a slightly different way.

\subsection{Metacirculants}
\label{sc:Metacirculants}

In \cite{MS}, Maru\v{s}i\v{c} and \v{S}parl considered tetravalent graphs which are properly called {\em weak metacirculant}, though we will  simply refer to them as {\em metacirculant} in this paper. A graph is metacirculant provided that it has a symmetry $\rho$ which acts on the $kn$ vertices as $k$ cycles of length $n$ and another symmetry $\sigma$ which normalizes $\langle \rho\rangle$ and permutes the $k$ $\rho$-orbits in a cycle of length $k$.  That paper  divides the $\frac{1}{2}$-transitive tetravalent metacirculant graphs into four classes.  The Type I graphs are the Power Spidergraphs $\PS(k,n,;r)$ and $\MPS(k, n; r)$. Papers \cite { M1, MS, S1, STG} completely determine which of these are $\frac{1}{2}$-transitive and which are dart-transitive. The Type II graphs
 are called ${\rm Y}$ there and will be called $\MSY$ here and in the census.  

 These also have been classified, in unpublished work.  See Section \ref{MSY} below.   The graphs of Type III we will call $\MC3$ in this census.  They have been studied with a few results.   See Section \ref{MC3}. The general Type IV graphs are very unruly.  Only a subclass of graphs, called ${\rm Z}$ in\cite{MS} and $\MSZ$ in the Census, have received much study and even here, there are few results.  See Section \ref{MSZ} for  a description of the graph.

\subsubsection{\MSY}
\label{MSY}
The graph $\MSY(m,n;r,t)$ has the diagram shown in Figure \ref{Fig:MSY}.  More precisely, its vertex set is $\ZZ_m \times \ZZ_n$, with two kinds of edges:
\begin{itemize}
\setlength{\itemsep}{0pt}
 \item[{\rm (1)}]  $(i, j) - (i, j+r^i)$ for all $i$ and $j$, and 
 \item[{\rm (2)}] $(i,j) - (i+1, j)$ for $0 < i<k-1$ and $(k-1, j) - (0, j+t)$ for all $j$.  
 \end{itemize}
 The graph is metacirculant if and only if $r^m = 1$  and $rt = t$.  Here, all equalities are equivalences mod $n$.  

\begin{figure}[hhh]
\begin{center}
\epsfig{file=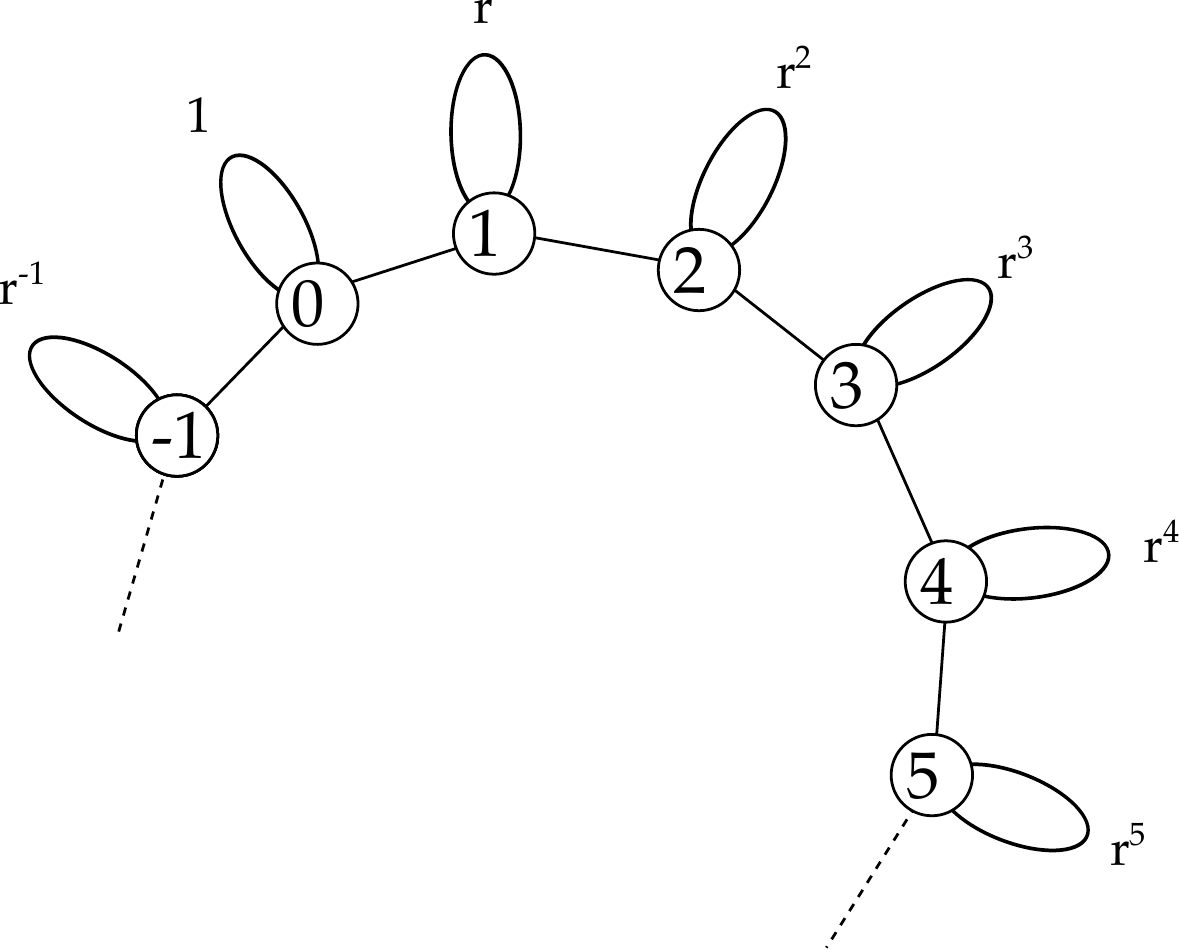,height=55mm}
\caption{Diagram for the graph $\MSY(k,n; r,t)$}
\label{Fig:MSY}
\end{center}
\end{figure}

The paper \cite{MS} proves that $\MSY(k, n;r, t)$ is metacirculant and  $\frac{1}{2}$-transitive if and only if it is isomorphic to one in which:

\begin{itemize}
\setlength{\itemsep}{0pt}
 \item[{\rm (1)}]   $n = dm$ for some integer $d \ge 3$,
 \item[{\rm (2)}]  $r^m = 1$,
 \item[{\rm (3)}] $r^2 \ne \pm1$,
 \item[{\rm (4)}]  $m(r-1) = t(r-1) = (r-1)^2 = 0$,
 \item[{\rm (5)}] $\langle m\rangle = \langle t\rangle$ in $\ZZ_n$,
 \item[{\rm (6)}] there is a unique $c$ in $\ZZ_d$ satifying $cm = t$ and $ct = m$, 
 \item[{\rm (7)}] there is a unique $k$ in $\ZZ_d$ satifying $kt = -km = r-1$, and 
 \item[{\rm (8)}] either $m \ne 4$ or $t \ne 2+2r$.
\end{itemize}

The paper \cite{MSY} shows that  $\MSY(m,n;r,t)$ is metacirculant and edge-transitive if and only if it is isomorphic to one in which $r^m = 1, rt = t$ and one of three things happens:

\begin{itemize}
\setlength{\itemsep}{0pt}
\item[{\rm (1)}] $m = (t, n)$ (then $t = sm, n=n'm$ and $(n', s) = 1$);
\item[{\rm (2)}] $r \equiv 1$ (mod $m$) and so $r = km+1$ for some $k$;
\item[{\rm (3)}] $st = m$;
\item[{\rm (4)}] $kt = -km$.
\end{itemize}
or
\begin{itemize}
\setlength{\itemsep}{0pt}
\item[{\rm (1)}] $m = (t, n)$ (then $t = sm, n=n'm$ and $(n', s) = 1$);
\item[{\rm (2)}] $r \equiv 1$ (mod $m$) and so $r = km+1$ for some $k$;
\item[{\rm (3)}] $st = -m$;
\item[{\rm (4)}] $kt = km$.
\end{itemize}
 or
 $[m,n,r,t]$ is one of these four sporadic examples: 
 $$[5, 11, 5, 0],\> [5, 22, 5, 11],\> [5,33, 16, 0],\> [5, 66, 31, 33].$$

\subsubsection{MSZ}
\label{MSZ}

The graph $\MSZ(m, n; k, r)$ has a diagram  isomorphic to the circulant graph $C_m(1, k)$.   It has vertex set $\ZZ_m\times\ZZ_n$.   The vertex $(i, j)$ is adjacent to $(i+1, j)$ and to $(i+k, j + r^i)$.
Very little is known about this family of graphs or the more general class of type IV metacirculants.

\subsubsection{MC3}
\label{MC3}
	The diagram for this family is a cycle of even length, with each node joined by two edges to the one opposite. 
The vertex set for $\Gamma = \MC3(m, n, a, b, r, t, c)$ (here, $m$ must be even) is $\ZZ_m\times\ZZ_n$.  Green edges connect each $(i, j)$ to $(i+c, j)$ for $i = 0, 1, 2, \dots, m-2$ and each $(m-1, j)$to $(0, j+t)$.  Red edges join each $(i, j)$ to $(i+\frac{m}{2}, j+ar^i)$ and $(i+\frac{m}{2}, j+br^i)$.  In order for this to be metacirculant, we must have $rt = t,r^m = 1, and \{a+t, b+t\} = \{-ar^\frac{m}{2}, -br^\frac{m}{2}\}$

 Because it has been shown  that each such graph which is $\half$-arc-transitive is also a $\PS,\MPS, \MSY$ or $\MSZ$, little attention has been given to it.  However, many MC3's are dart-transitive and many are LR structures (see section \ref{sc:Cycle Decompositions}).  Each of the following families is such an example:
\begin{enumerate}
\setlength{\itemsep}{0pt}
\item   $m$ is divisible by 4, $n$ is divisible by 2, $r^2 = \pm 1, a= 1, b = -1, t = 0$,
\item $m$ is  not divisible by 4, $n$ is divisible by 4, $r^2 =  1, a= 1, b = -1, t = \frac{n}{2}$,
\item $n$ is divisible by 4, $r^2=  1, a= 1, b = n/2 -1, t = n/2$, or 
\item $m$ is  not divisible by 4, $n$ is divisible by 4, $r^2 =  1, a= 1, b =n/2 -1, 2t = 0$.
\end{enumerate}

These were found and proved to be LR strucures (see section \ref{sc:Cycle Decompositions}) by Ben Lantz \cite{MC3}, and there are  LR examples not covered by these families.  Further, many of the MC3 graphs are dart-transitive.  There are many open questions about this family.

\subsection {Other diagrams}
\label{sc:OD}
A number of other diagrams have been found to give what appears to be an infinite number of examples of edge-transitive graphs.  The first of these is the Long Propellor,$\LoPr_n(a, b, c, d)$, shown in Figure \ref{Fig:LoPr}.
\begin{figure}[hhh]
\begin{center}
\epsfig{file=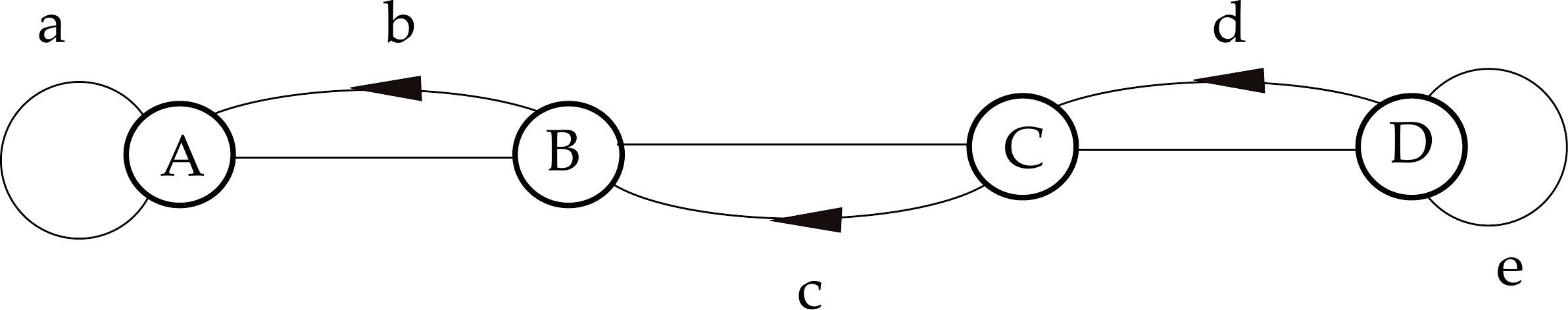,height=20mm}
\caption{Diagram for the graph $\LoPr_n(a, b, c, d, e)$}
\label{Fig:LoPr}
\end{center}
\end{figure}

Next is the Wooly Hat, $\WH_n(a, b, c, d, e)$, shown in Figure \ref{Fig:wooly}.  This diagram gives no edge-transitive covers, but it does yield a family of LR structures (see Section \ref{sc:Cycle Decompositions}), as yet unclassified.

\begin{figure}[hhh]
\begin{center}
\epsfig{file=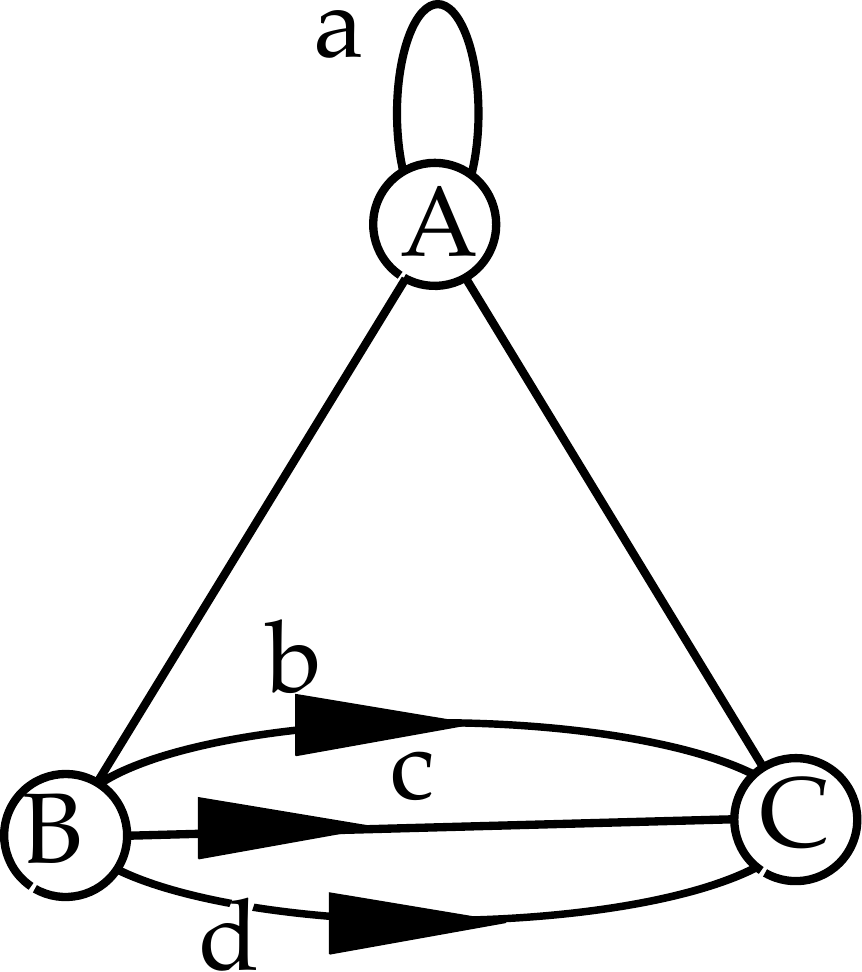,height=35mm}
\caption{Diagram for the graph $\WH_n(a, b, c, d, e)$}
\label{Fig:wooly}
\end{center}
\end{figure}
The Kitten Eye, $\KE_n(a, b, c, d, e)$, shown in Figure \ref{Fig:KE}, has dart-transitive covers.
\begin{figure}[hhh]
\begin{center}
\epsfig{file=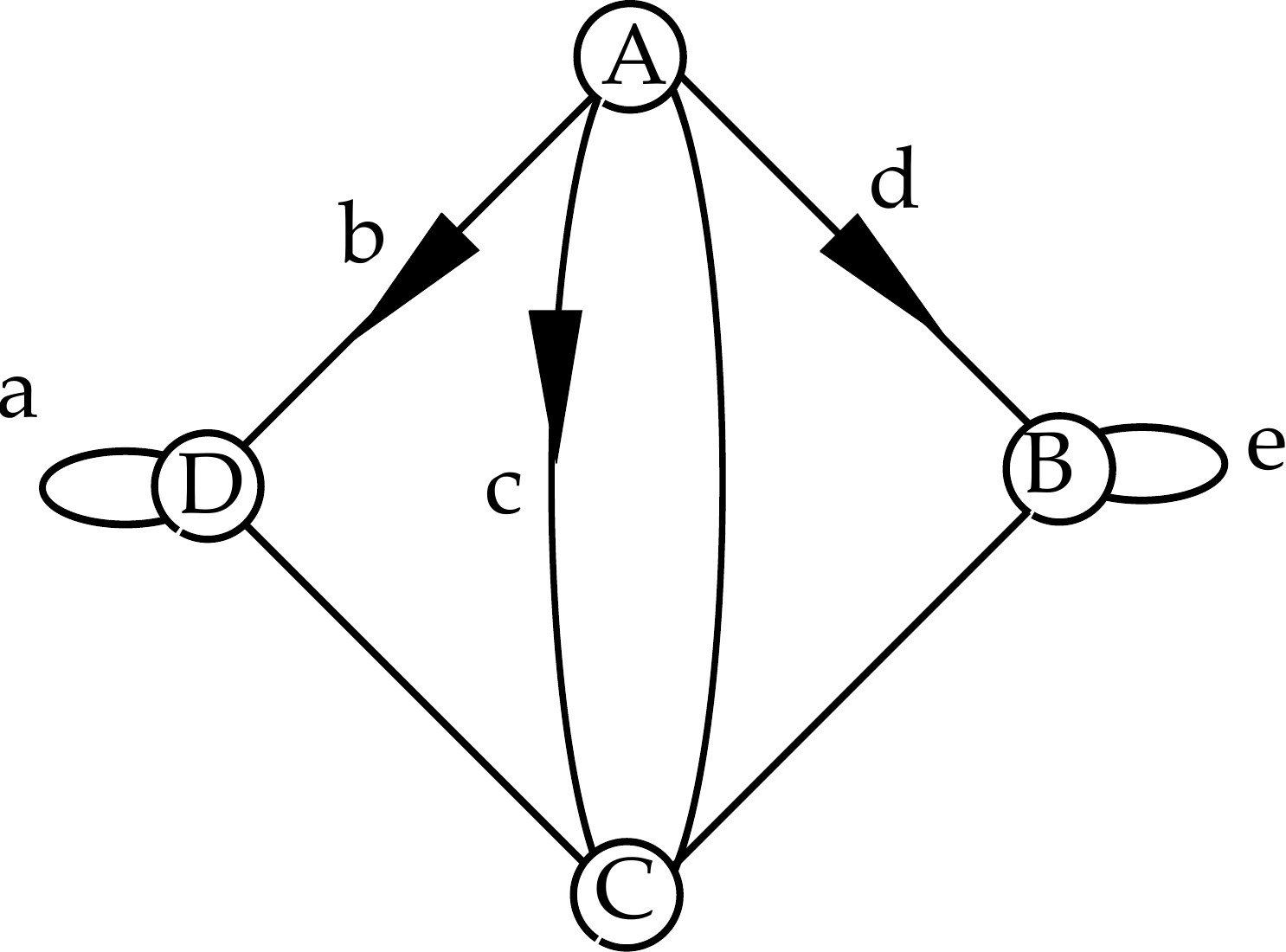,height=35mm}
\caption{Diagram for the graph $\KE_n(a, b, c, d, e)$}
\label{Fig:KE}
\end{center}
\end{figure}

The Curtain,  $\Curtain_n(a, b, c, d, e)$, shown in Figure \ref{Fig:Curtain}, has both dart-trasitive and LR covers.
\begin{figure}[hhh]
\begin{center}
\epsfig{file=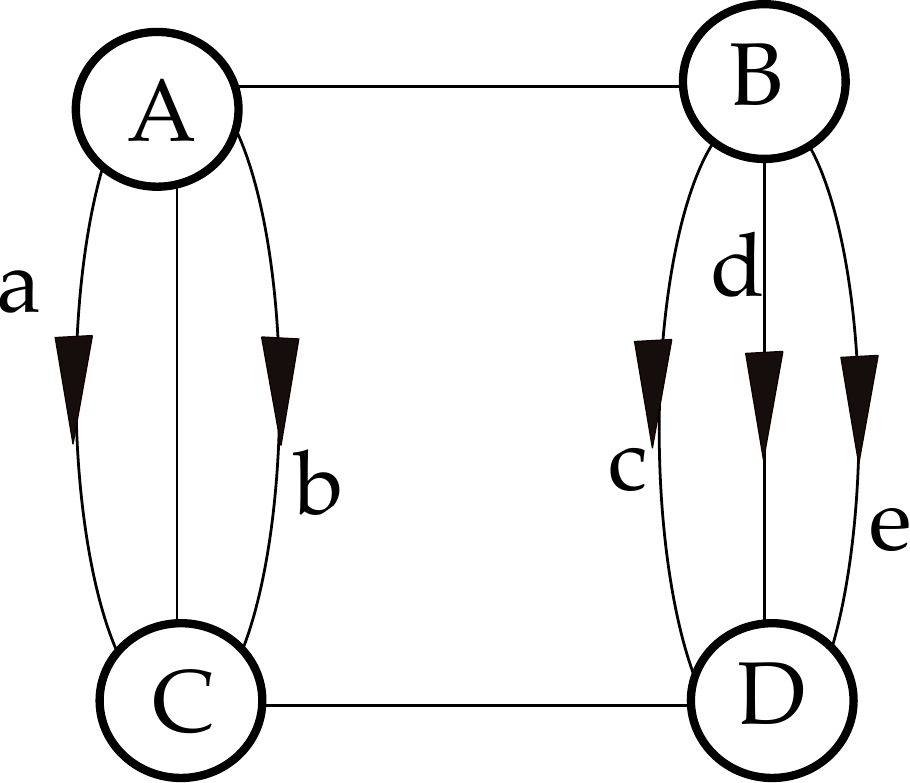,height=30mm}
\caption{Diagram for the graph $\Curtain_n(a, b, c, d, e)$}
\label{Fig:Curtain}
\end{center}
\end{figure}

\section{Praeger-Xu Constructions}
\label{sec:PXC}

	The graph called $C(2, n,k)$ in \cite{PX} is also described in \cite {GP} in two different ways.  In this Census, we will name the graph $\PC(n, k)$.  In order to describe the graph, we need some notation about bit strings.  A {\em bit string} of length $k$ is the concatenation of $k$ symbols, each of them a '0' or a '1'.  For example $x$ = 0011011110 is a bit string of length 10.  If $x$ is a bit string of length $k$, then $x_i$ is its $i$-th entry, and $x^i$ is the string identical to $x$ in every place except the $i$-th.  Also $1x$ is the string of length $k+1$ formed from $x$ by placing a '1' in front; similar definitions hold for the $(k+1)$-strings $0x, x0, x1$.  Finally, the string $\bar{x}$ is the reversal of $x$.

	The vertices of the graph $\PC(n, k)$ are ordered pairs of the form $(j,x)$, where $j \in \ZZ_n$ and $x$ is a bit string of length $k$.  Edges are all pairs of the form  $\{(i,0x), (i+1, x0)\}, \{(i,0x), (i+1, x1)\}, \{(i,1x), (i+1, x0)\}, \{(i,1x), (i+1, x1)\}$, where $x$ is any bit string of length $k-1$.

	We first wish to consider some symmetries of this graph.  First we note  $\rho$ and $\mu$ given by  $(j, x)\rho = (j+1, x)$ and $(j,x)\mu = (-j, \bar{x})$   These are clearly symmetries of the graph and act on it as $D_r$.

	For $b \in \ZZ_n$, we define the symmetry $\sigma_b$ to be the permutation which interchanges $(b-i, x)$ with $(b-i, x^i)$ for $i = 1, 2, 3, \dots , k$ and leaves all other vertices fixed.  If $ n > k$, then the symmetries $\sigma _0, \sigma_1, \dots, \sigma_{n-1}$ commute with each other  and thus generate an elementary abelian group of order $2^n$, while the  symmetries
$\rho, \mu$ and $\sigma_b$'s generate a semidirect product $\ZZ_2^n\rtimes D_n$ of order $n2^{n+1}$. Unless
$n=4$, this is also the full symmetry group of the graph.

	The Praeger-Xu graphs generalize two families of graphs:  $\PC(n,1) = \W(n, 2)$ and $\PC(n, 2) = \RW_{2n}(n+2, n+1)$.

\section{Gardiner-Praeger Constructions}
\label{sc:GP}

The paper \cite {GP} constructs two families of tetravalent graphs whose groups contain large normal subgroups such that the factor graph is a cycle.  The first is $C^{\pm 1}(p,st,s)$,  and the second is $C^{\pm e}(p,2st,s)$.  In the Census, we use a slight generalization of both, which we notate $\CPM(n, s, t, r)$, where $n$ is any integer at least 3, $s$ is an integer at least 2,  $t$ is a positive integer, and $r$ is a unit mod $n$.
	We first form the digraph $\CPM[n, s, t, r]$.  Its vertex set is $\ZZ_n^s\times\ZZ_{st}$.  Directed edges are of the form $((x, i), (x\pm r^ie_j, i+1))$, where $j$ is $i$ mod $s$, and $e_j$ is the $j$-th standard basis vector for $\ZZ_n^s$.  If $r^{st}$ is $\pm 1$ mod $n$, then $\CPM[n, s, t, r]$ is a semitransitive orientation for its underlying graph $\CPM(n, s, t, r)$.  When the graph is not connected, we re-assign the name $\CPM(n, s, t, r)$ to the component containing $(0, 0)$.  If $n$ is odd, then the graph has $stn^s$ vertices.  If $n$ is even then it has $st(\frac{n}{2})^s$ vertices if $t$ is even and twice that many if $t$ is odd.

	Some special cases are known:  For $r= 1, \CPM(n, s, t, 1)\cong C^{\pm 1}(n,st,s)$.  If $r^2 = 1$, then $\CPM(n, s, 2t, r)$ is $C^{\pm r}(p,2st,s)$.  When $s = 1, \CPM(n, 1, t, r)\cong \PS(n,t;r)$.  If $s=1$ and $t=4$, then the graph is a Wreath graph.  When $s = 2$ and $t$ is 1 or 2, then the graph is toroidal.  Other special cases are conjectured:

The convention in the following conjectures is that $q$ is a number whose square is one mod $m$ and p is the parity function:

\begin{equation}
p(t) = \left\lbrace{\begin{array}{ll}
1  &\mbox{if t is odd} \\
2  &\mbox{if t is even}\\
\end{array}}\right.  \end{equation}

With that said, we believe that:

\begin{enumerate}
\item If $t$ is not divisible by 3, then $\CPM(3, 2, t, 1) \cong \PS(6, m; q)$  where $m = 3t$.
\item If $t$ is not divisible by 5, then $\CPM(5, 2, t, 1) \cong \CPM(5, 2, t, 2) \cong \PS(10, m; q)$  where $m= 5tp(t)$.
\item If $t$ is not divisible by 3, then $\CPM(6, 2, t, 1)  \cong \PS(6, m; q)$  where $m=\frac{12t}{p(t)}$.
\item If $t$ is not divisible by 4, then $\CPM(8, 2, t, 1) \cong \CPM(8, 2, t, 3) \cong \MPS(8, m; q)$  where $m= \frac{16t}{p(t)}$.
\item For all $s$, $\CPM(4, s, t, 1)\cong \PC(\frac{2st}{p(t)}, s)$.
\end{enumerate}

\section{Graphs $\Gamma^\pm$ of Spiga, Verret and Poto\v{c}nik}
\label{sc:Cubic}

It was proved in \cite{lost} that a tetravalent graph $\Gamma$ whose automorphism group $G$ is dart-transitive
is either $2$-arc-transitive (and then $|G_v| \le  2^43^6$), a $\PX$-graph (see Section~\ref{sec:PXC}), one of eighteen 
exceptional graphs, or it satisfies the inequality 
$$|V(\Gamma)| \ge 2|G_v| \log_2(|G_v|/2). \eqno{(*)}$$
 This result served as the basis
for the construction of a complete list of dart-transitive tetravalent graphs (see \cite{CubicCensus} for details). 

Moreover, in \cite{gamma} it was proved that a graph attaining the bound $(*)$ has $t2^{t+2}$ vertices for some $t\ge 2$
and is isomorphic to one of the graphs
$\PPM(t,\epsilon)$, for $\epsilon \in \{0,1\}$,
 defined below as \emph{coset graphs} of certain groups $G_t^+$ or $G_t^-$. Recall that the 
coset graph $\Cos(G,H,a)$ on a group $G$ relative to a subgroup $H\le G$ and an element $a\in G$ 
is defined as the graph with vertex set 
 the set of right cosets $G/H = \{Hg \mid g \in G\}$ and with edge set the set $ \{ \{Hg, Hag\} \mid g \in G\}$.

 For $\epsilon \in \{0,1\}$ and $t\ge 2$, let $G_t^\epsilon$  be the group defined as follows:
$$
\begin{array}{ll}
G_t^\epsilon =\langle x_0,\ldots,x_{2t-1},z, a, b \> \> \mid & x_i^2=z^2 = b^2 = z^\epsilon a^{2t}= (ab)^2  = 1, \\
                                                                        & [x_i,z] = 1 \textrm{ for }0\leq i\leq 2t-1, \\
                                  & [x_i,x_j]  =  1 \textrm{ for }|i-j|\neq t, \\
&     [x_i,x_{t+i}]  =  z \> \textrm{ for  }  \> 0\leq i\leq t-1, \\
&    x_i^a = x_{i+1}\> \hbox{  for } \> 0\leq i\leq 2t-1,\> \\
&    x_i^b = x_{t-1-i}\> \hbox{  for } 0 \>\leq i\leq 2t-1 \rangle.
\end{array}
 $$ 
 In either group, we let $H = \langle x_0,\ldots,x_{t-1},b\rangle$, and define graphs
 $$\PPM(t,\epsilon) = \Cos(  G_t^\epsilon, H, a),$$
 denoted $\Gamma_t^+$ (for $\epsilon = 0$) and $\Gamma_t^-$ (for $\epsilon = 1$) in \cite{gamma}.
We should point out that a graph $\Gamma =\PPM(t,\epsilon)$ is a $2$-cover of the Praeger-Xu graph $\PX(2t,t)$. Furthermore, 
 the girth of $\PPM(t,\epsilon)$ is generally $8$, the only exceptions being that $\PPM(2,0)$ has girth 4, and $\PPM(3,0)$ has girth 6.
Finally, $\PPM(2,0) \cong \PX(4,3)$, while in all other cases the graph $\Gamma$ is not isomorphic to a $\PX$ graph.

\section{From Cubic graphs}
\label{sc:Cubic}

	In this section, we describe five constructions, each of which constructs a tetravalent graph from a smaller cubic (i.e., trivalent) graph in such a way that the larger graph inherits many symmetries from the smaller graph.  Throughout this section, assume that $\Lambda$ is a  cubic graph, and that it is dart-transitive.  Our source of these graphs is Marston Conder's census of symmetric cubic graphs of up to 10,000 vertices  \cite{MC3}.

\subsection{Line graphs}
	The {\em line graph} of $\Lambda$ is a graph $\Gamma = L(\Lambda)$ whose vertices are, or correspond to, the edges of $\Lambda$.  Two vertices of $\Gamma$ are joined by an edge exactly when the corresponding edges of $\Lambda$ share a vertex.  Every symmetry of $\Lambda$ acts on $\Gamma$ as a symmetry, though $\Gamma$ may have other symmetries as well.  Clearly, if $\Lambda$ is edge-transitive, then $\Gamma$ is vertex-transitive.   If $\Lambda$ is dart-transitive, then $\Gamma$ is edge-transitive, and if $\Lambda$ is 2-arc-transitive, then $\Gamma$ is dart-transitive.

\subsection{Dart Graphs}
	The {\em Dart Graph} of  $\Lambda$ is a graph $\Gamma = \DG(\Lambda)$ whose vertices are, or correspond to, the darts of $\Lambda$.  Edges join a dart $(a,b)$ to the dart $(b, c)$ whenever $a$ and $c$ are neighbors of $b$.  Clearly, $ \DG(\Lambda)$ is a two-fold cover of $L(\Lambda)$.


\subsection{Hill Capping}
	For every vertex $A$ of $\Lambda$, we consider the {\em symbols} $(A, 0), (A, 1)$, though we will ususally write them as $A_0, A_1$.   Vertices of $\Gamma = \HC(\Lambda)$ are all unordered pairs $\{A_i, B_j\}$ of symbols where $\{A, B\}$ is an edge of $\Lambda$.  Edges join each vertex $\{A_i, B_j\}$ to $\{B_j, C_{1-i}\}$  where $A$ and $C$ are neighbors of $B$. 

	If $\Lambda$ is bipartite and 2-arc-transitive then $\Gamma$ is dart-transitive. If $\Lambda$ is bipartite and  {\em not} 2-arc-transitive then $\Gamma$ is semisymmetric.  If $\Lambda$ is  {\em not} bipartite and is 2-arc-transitive then $\Gamma$ is $\frac{1}{2}$-arc-transitive.

	$\HC(\Lambda)$ is clearly a fourfold cover of $L(\Lambda)$; it is sometimes but not always a twofold cover of $\DG(\Lambda)$.  The Hill Capping is described more fully in \cite{HW}.

\subsection{3-arc graph}
\label{ssc:TAG}
	The three-arc graph of $\Lambda$, called $A_3(\Lambda)$ in the literature  \cite{KZ} and called $\TAG(\Lambda)$in the Census, is a graph whose vertices are the darts of $\Lambda$, with $(a, b)$ adjacent to $(c, d)$ exactly when $[b, a, c, d]$ is a 3-arc in $\Lambda$.  Thus, $a$ and $c$ are adjacent, $b\ne c$ and $a\ne d$.  This graph is dart-transitive if $\Lambda$ is 3-arc transitive.  

\section{Cycle Decompositions}
\label{sc:Cycle Decompositions}
	A {\em cycle decomposition} of a tetravalent graph $\Lambda$ is a partition $\C$ of its edges into cycles.  Every edge belongs to exactly one cycle in $\C$ and each vertex belongs to exactly two cycles of $\C$.  $\Aut(\C)$ is the group of all symmetries of $\Lambda$ which preserve $\C$.   One possibility for a symmetry is a {\em swapper}.    If $v$ is a vertex on the cycle $C$, a $C-${\em swapper} at $v$ is a symmetry which reverses $C$ while fixing $v$ and every vertex  on the other cycle through $v$.

 If $\C$ is a cycle decomposition of $\Lambda$, the {\em partial line graph} of $\C$, written $\PP(\C)$ and notated ${\rm PL}(\C)$ in the Census, is a graph $\Gamma$ whose vertices are (or correspond to) the edges of $\Lambda$, and whose edges are all $\{e, f\}$ where $e$ and $f$ are edges which share a vertex but belong to different cycles of $\C$.

	Because $\Aut(\C)$ acts on $\Gamma$ as a group of its symmetries, the partial line graph is useful for constructing  graphs having a large symmetry group.     Almost all tetravalent dart-transitive graphs have cycle decompositions whose symmetry group is transitive on darts.  These are called ``cycle structures'' in \cite{PWCD}.

	If $\Lambda$ is $\frac{1}{2}$-arc-transitive, then it has a cycle decomposition $\A$ into 'alternating cycles' \cite{MP}.  If the stabilizer of a vertex has order at least 4, then $\PP(\A)$ has a $\half$-arc-transtive action and may actually be $\frac{1}{2}$-arc-transitive.  

	Many graphs in the census are constructed from smaller ones using the partial line graph.  Important here are the Praeger-Xu graphs.  Each $\PC(n, k)$ has a partition $\C$ of its edges  into 4-cycles  of the form:$(i,0x), (i+1, x0), (i,1x), (i+1, x1)$.  Then $\PP(\C)$ is $\PC(n, k+1)$.  A special case of this is that family (b) of Rose Window graphs is $\PP$ applied to family (a), the wreath graphs.

	The toroidal graphs have a cycle decomposition in which each cycle consists entirely of vertical edges or entirely of horizontal edges.  The partial line graph of this cycle decomposition is another toroidal graph.  For the rotary case,  $\PP(\{4, 4\}_{b, c}) = \{4, 4\}_{b+c, b-c}$.    The other two cases of toroidal maps do not have edge-transitive partial line graphs.

	It may happen that a cycle decomposition $\C$ is a 'suitable LR structure' \cite{LR1}; this means that $\C$ has a partition into $\R$ and $\G$ (the 'red' and the 'green' cycles) such that every vertex belongs to one cycle from each set, that the subgroup of $\Aut(\C)$ which sends $\R$ to itself is transitive on the vertices of $\Lambda$, that $\Aut(\C)$ has all possible swappers, that no element of $\Aut(\C)$ interchanges $\R$ and $\G$ and, finally, that no 4-cycles alternates between $\R$ and $\G$.  With all of that said, if $\C$ is a suitable LR structure, then $\PP(\C)$ is a semisymmetric tetrtavalent graph in which each edge belongs to a 4-cycle.  Further, every such graph is constructed in this way. \cite{LR1}

\section{Some LR Structures}
\label{sc:SomeLR}
  
\subsection{Barrels}
	The barrels are the most common of the suitable LR structures.  The standard barrel is $\Br(k, n; r)$, where $k$ is an even integer $\ge 4$, $n$ is an integer $\ge 5$ and $r$ is a number mod $n$ such that $r^2 = \pm 1$ (mod $n$) but $r \ne \pm 1$ (mod $n$).  The vertex set is $\ZZ_k\times\ZZ_n$.  Green edges join each $(i, j)$ to $(i, j+r^i)$.  Red edges join each $(i, j)$ to $(i+1, j)$.   

The mutant barrel is M$\Br(k, n; r)$, where $k$ is an even integer $\ge 2$, $n$ is an {\em even} integer $\ge 8$ and $r$ is a number mod $n$ such that $r^2 = \pm 1$ (mod $n$) but $r \ne \pm 1$ (mod $n$).  The vertex set is $\ZZ_k\times\ZZ_n$.  Green edges join each $(i, j)$ to $(i, j+r^i)$.  Red edges join each $(i, j)$ to

$$\begin{cases}(i+1, j)&\text{if $i \ne k-1$ }\\ (0, j+\frac{n}{2})&\text{if $i = k-1$}\\ \end{cases}.$$

\subsection{Cycle Structures}
	If $\Lambda$ is a tetravalent graph admiting a cycle structure $\C$, we can form an LR struture from it in two steps:

1. Replace each vertex $v$ with two vertices, each incident with the two edges of one of the two cycles in $\C$ containing $v$; think of these as green edges.  Join the two vertices corresponding to $v$ with two parallel red edges.

2.  Double cover this structure.  We assign weights or voltages to red edges so that each pair has one 0 and one 1.  Voltages for green edges are assigned in one of two ways:  (0) every green edge gets voltage 0 or (1) one edge in each green cycle gets voltage 1, and the rest get 0.  The double covers corresponding to these two assignments are called $\CS(\Lambda, \C, 0)$ and  $\CS(\Lambda, \C, 1)$, respectively,  and they are, in most cases, suitable LR structures, as shown in \cite{LRcomb}.

\subsection{Bicirculants}
\label{ss:BC}
	Consider a bicoloring of the edges of the bicirculant $\BC_n(0, a, b, c)$ with green edges linking $A_i$ to $B_i$ and $B_{i+a}$ and  red edges linking $A_i$ to $B_{i+b}$ and $B_{i+c}$.  We call this coloring $\BC_n(\{0, a\}, \{b, c\})$.  The paper \cite{LRalg} shows several cases in which  $\BC_n(\{0, a\}, \{b, c\})$ is a suitable LR structure:
\begin{enumerate}
\item $a = 1-r,    b=1,    c=s$, where 
$r,s \in \ZZ_n^*\setminus\{-1,1\},\> r^2 = s^2 = 1, \> r\not \in \{-s,s\}, 
                  \> \hbox{ and } \> (r-1)(s-1) = 0. $
\item $n = 2m, a = m, b = 1,  c\in \ZZ_{2m}\setminus \{1,-1, m+1, m-1\}$ such that $c^2 \in \{1, m+1\}$.
\item $n = 4k, a = 2k, b= 1, c = k+1,$ for $k \ge 3$
\end{enumerate}
Moreover, it is conjectured in that paper that every suitable $\BC_n(\{0, a\}, \{b, c\})$ is isomorphic to at least one of these three.

\subsection{\MSY's and \MSZ's}
	The graph $\MSY(m,n;r,t)$ has an LR structure, with edges of the first kind being red and those of the second kind being green, if and only if $2t = 0$ and $r^2 = \pm1$.  Many examples of $\MSZ$ graphs being suitable LR structures are known, but no general classification has been attempted.

\subsection{Stack of Pancakes}
	The structure is called $\SoP(4m, 4n)$.  Let $r = 2n+1$.  The vertex set is $\ZZ_{4m} \times \ZZ_{4n} \times \ZZ_2$.  Red edges join $(i,j,k)$ to $(i, j\pm r^k, k)$; for a fixed $i$ and $j$,  green edges join the two vertices $(2i,j,0)$ and $(2i,j,1)$ to the two vertices $(2i+1,j,0)$ and $(2i+1, j, 1)$ if $j$ is even, to the two vertices $(2i-1,j,0)$ and $(2i-1, j, 1)$ if $j$ is odd.

	The paper \cite{LRcomb} shows that this is a suitable LR structure for all $m$ and $n$, and that the symmetry group of it and of its partial line graph, can have arbitrarily large vertex-stabilizers.\\

\subsection{Rows and Columns}
	
The LR structure  ${\rm RC}(n, k))$ has as vertices all ordered pairs $(i, (r, j))$ and $((i, r), j)$,
 where $i$ and $j$ are in $\ZZ_n$, and $r$ is in $\ZZ_k$, where $k$ and $n$are integers at least 3.
 Green edges join $(i, (r, j))$ to $(i\pm 1, (r, j))$ and $((i, r), j)$ to $((i, r), j\pm 1)$,
 while red edges join $(i, (r, j))$ to $((i, r \pm 1), j)$ and so $((i, r), j)$ to $(i, (r \pm 1, j))$.\\

This structure is referred to in both \cite{LRalg} and \cite{LRcomb}.

\subsection{Cayley constructions}
	Suppose a group $A$ is generated by two sets,  $R$ and $G$,  of size two, neither containing the identity, and each containing the inverse of each of its elements.  Then we let $\Cay(A;R, G)$ be the structure whose vertex set is $A$, whose red edges join each $a$ to $sa$ for $s\in R$ and whose green edges join each $a$ to $sa$ for $s\in G$.  The paper \cite{LRalg} shows that if $A$ admits two automorphisms, one fixing each element of $R$ but interchanging the two element of $G$ and the other vice versa, then $\Cay(A;R, G)$ is an LR structure.  The condition $RG \neq GR$ is equivalent to the structure not having alternating 4-cycles.
	
	The most frequently occuring examples in the case where $A$ is the dihedral group $D_n$.  One family of this type is the first group of bicirculants $BC_n(\{0, 1-r\}, \{1, s\})$ shown in in subsection \ref{ss:BC}.

	The paper \cite{LRalg} shows several other algebraically defined structures.  First, there are examples for the group $D_n$ where the swappers do not arise from group automorphisms.  Second, there is a Cayley construction for the structure $RC(n,k)$ of the previous subsection.

And the body of the paper shows six  'linear' constructions in which $A$ is an extension of some $\ZZ_n^k$:

\begin{construction}
For $n$ and $k$ both at least 3, let $A$ be a semidirect product of $\ZZ_n^k$ with the group generated by the permutation $\sigma= (1  2  3  \dots  k-1   k)$ acting on the coordinates.  Let $e_1$ be the standard basis element $(1   0  0 \dots  0)$,  let $R = \{e_1, -e_1\}$ and $G = \{\sigma, \sigma^{-1}\}$.  We define the LR structure $\AffLR(n, k)$ to be $\Cay(A; R, G)$.
\end{construction}

\begin{construction}
Let $\ProjLR(k,n)$ be $\AffLR(k,n)$ factored out by the cyclic group generated by $(1, 1, \dots, 1)$.
\end{construction}

\begin{construction}
Let $\ProjLR^{\circ}(2k,n)$ be $\AffLR(2k,n)$ factored out by the group generated by all $d_i = e_i-e_{i+k}$, where $e_i$ is the standard basis element having a 1 in position $i$ and zeroes elsewhere.
\end{construction}

\begin{construction}
Let $A$ be a semidirect product of $\ZZ_2^{2k}$ with the group generated by the permutation $\gamma = (1, 2, 3, \dots, k)(k+1, k+2, \dots, 2k)$. acting on the coordinates.  Let $R = \{e_1, e_{k+1}\}$ and $G = \{\gamma, \gamma^{-1}\}$.  We define the LR structure $\AffLR_2(k)$ to be $\Cay(A; R, G)$.
\end{construction}

 Let $d_1$ be the $2k$-tuple in which the first $k$ 
entries are 1 and the last $k$ entries are 0; let $d_2$ be the $2k$-tuple in which the
first $k$ entries are 0 and the last $k$ entries are 1; let $d$ =  $d_1+d_2$. 

\begin{construction}
Let $\ProjLR'(k)$ be $\AffLR_2(k)$ factored out by the  group generated by $d_1$ and $d_2$.
\end{construction}

\begin{construction}
Let $\ProjLR"(k)$ be $\AffLR_2(k)$ factored out by the  group generated by $d$.
\end{construction}

The paper \cite{LRalg} proves that all six of these constructions lead to suitable LR structures (except for a few cases).

\section{Base Graph-Connection Graph}
\label{sc:BGCG}
	Let $B$ be any tetravalent graph, and define $B^*$ to be the graph formed from $B$ by replacing every edge with a path of length 2.  Think of the tetravalent vertices of $B^*$ as being black and the degree-2 vertices as white.  Suppose that $\Gamma$ is a tetravalent graph whose edges can be partitioned into subgraphs each isomorphic to $B^*$; call these subgraphs 'blocks'.  Define a graph $C$ to have one vertex for each block, with an edge connecting any two vertices corresponding to blocks that share a white vertex.  Call $B$ the ``base graph'' and $C$ the connection graph.  We call $\Gamma$ a ``BGCG of $B$ and $C$''.  Many of the graphs in the census are BGCG of some smaller $B$ and some $C$.  In \cite{BGCG}, we give constructions for the cases in which $C$ has at most two vertices.  
These constructions are sumarized in the next paragraph.  Also, a complete construction is known in the case when $C$ is the $n$-cycle $C_n$.   A simple special case  of this construction  is shown in the second paragraph below.  No truly general technique exists at the moment.

	If $B$ is a tetravalent dart-transitive graph, a {\em dart-transitive pairing} $\beta$ of $B$ is a partition of its edges into sets  of size 2 in such a way that the subgroup $G$ of $\Aut(B)$ which preserves the partition is large enough to be transitive on darts.  Given such a $\beta$, let $\kappa$ be the permutation of edges which sends each edge to the other edge of the same color.   Then $\BGCG(B, K_1,\beta)$ is the result of identifying, in one copy of $B^*$ each white vertex corresponding to an edge $e$ with the one corresponding to $e^\kappa$.  Similarly, $\BGCG(B, K_2,\beta)$ is formed from two copies of $B^*$, identifying, in one of the two copies each white vertex corresponding to an edge $e$ with the one corresponding to $e^\kappa$ in the other copy.  There are also some non-involutary permutations $\kappa$  which can be used.  The search for those is not yet implemented.

Now suppose that $\beta$ is a dart-transitive pairing with group $G$ and that there is a partition $\{\R, \G\}$ of the edges into two `colors' ($\R$ = `red', $\G$= `green') which is also invariant under $G$.   Suppose further that each pair in $\beta$ meets both $\R$ and $\G$. Then we can form $\BGCG(B, C_k, \{\beta, \{\R, \G\}\})$ by making copies $B_0, B_1, B_2, \dots, B_{k-1}$ of $B^*$.  For each pair $\{e, f\}$ in $\beta$, we identify the  copy of the green one in $B_i$ with the copy of the red one in $B_{i+1}$.   Alternatively, if $k$ is even, we can form $\BGCG(B, C_k, \{\beta, \{\R, \G\}\}')$ similarly, but identifying green $e$ in $B_i$ with $e$ in $B_{i+1}$ when $i$ is even and red $f$ in $B_i$ with $f$ in $B_{i+1}$ when $i$ is odd.  In each case, if $B$ is dart-transitive, then the constructed graph is edge-transitive, and is often semisymmetric.

\section{From regular maps}
\label{sc:maps}

A {\em map} is an embedding of a graph or multigraph on a compact connected surface such that each component of the complement of the graph (these are called {\em faces}) is topologically a disk.  A {\em symmetry} of a map is a symmetry of the graph which extends to a homeomorphism of the surface.  A map $\M$ is {\em rotary} provided that for some face and some vertex of that face, there is a symmetry $R$ which acts as rotation one step about the face and a symmetry $S$ which acts as rotation one step about the vertex.  A map $\M$ is {\em reflexible} provided that it is rotary and has a symmetry $X$ acting as a reflection fixing that face and vertex. If $\M$ is rotary but not reflexible, we call it {\em chiral}.    See \cite{HW} for more details.

	If $\M$ is rotary, its symmetry group, $\Aut(\M)$, is transitive on faces and on vertices.  Thus all faces have the same number $p$ of sides and all vertices have the same degree $q$.  We then say that $\M$ has {\em type} $\{p, q\}$.

\subsection{Underlying graphs}
	The underlying graph of a rotary map $\M$ is called ${\rm UG}(\M)$ and is always dart-transitive.  If $q = 4$, it belongs in this census.

\subsection{Medial graphs}
	The vertices of the medial graph, ${\rm MG}(\M)$, are the edges of $\M$.  Two are joined by an edge if they are consecutive in some face (and so in some vertex).  If $\M$ is rotary, ${\rm MG}(\M)$ is always edge-transitive.  If $\M$ is reflexible or if it is self dual in one of two ways, then ${\rm MG}(\M)$ is dart-transitive.  If not, it is quite often, but not quite always, $\frac{1}{2}$-transitive.  No one seems to know a good criterion for this distinction.

\subsection{Dart graphs}
	The vertices of the dart graph, $\DG(\M)$, are the darts of $\M$.  Two are joined by an edge if they are head-to-tail consecutive in some face.  The graph $\DG(\M)$ is a twofold cover of ${\rm MG}(\M)$ and is often the medial graph of some larger rotary map. It can be dart-transitive or $\frac{1}{2}$-transitive; again, no good criterion is known.

\subsection{HC  of maps}
	The Hill Capping of a rotary map $\M$ is defined in a way completely analogous to the capping of a cubic graph $\Lambda$:  we join  $\{A_i, B_j\}$ to $\{B_j, C_{1-i}\}$  where $A, B$ and $C$ are vertices which are consecutive around some face.  
The graph $\HC(\M)$ is a 4-fold covering of ${\rm MG}(\M)$ and can be dart-transitive or semisymmetric or $\frac{1}{2}$-transitive or even not edge-transitive.

\subsection{XI of maps}

Suppose that $\M$ is a rotary map of type $\{p, q\}$ for some even $q = 2n$.  Then each corner of the map (formed by two consecutive edges in one face)  is opposite at that vertex to another corner;  we will call such a pair of corners an 'X'.  As an example, consider Figure \ref{Fig:XI}, which shows one vertex, of degree 6, in a map.  The  X's are pairs $a, b, c$ of opposite corners.

\begin{center}
\begin{figure}[hhh]
\begin{center}
\epsfig{file=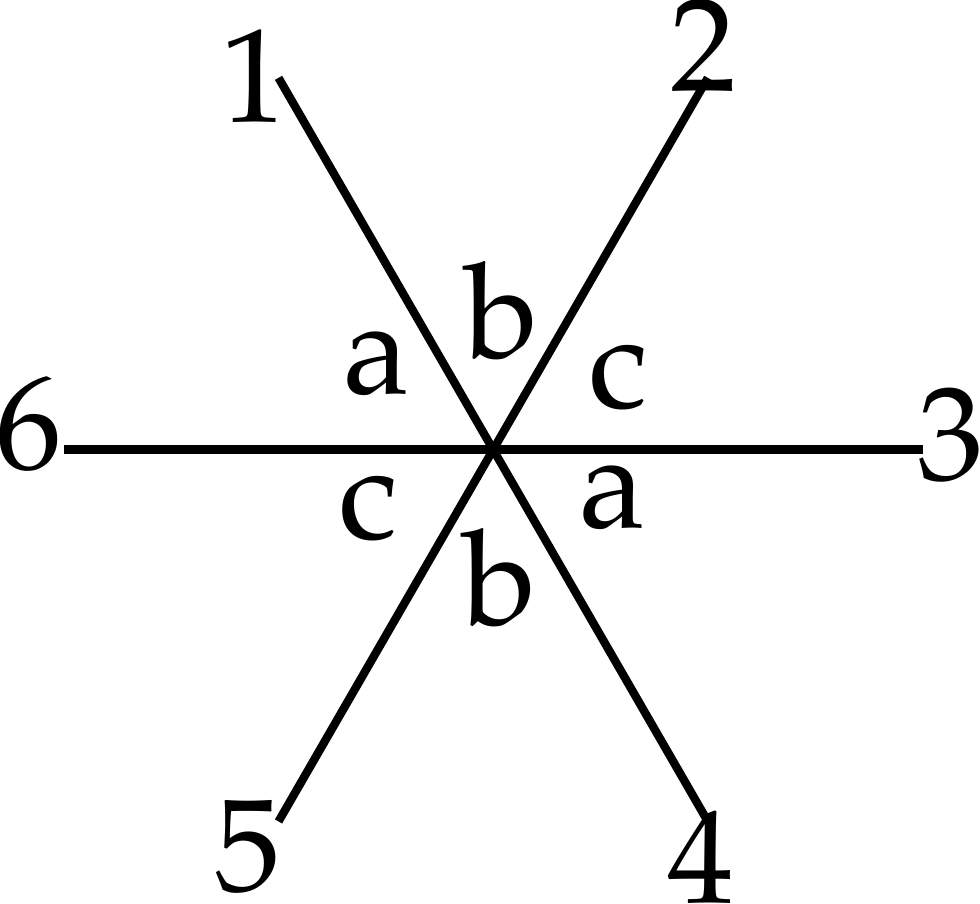,height=35mm}
\caption{X's and I's in a map}
\label{Fig:XI}
\end{center}
\end{figure}
\end{center}

We form the bipartite graph $\XI(\M)$  in this way:  The black vertices are the X's, the white vertices are the edges of $\M$, and edges of $\XI(\M)$ are all pairs $\{x, e\}$ where $x$ is an X and $e$ is one of the four edges of $x$.  Continuing our example, the black vertex $a$ is adjacent to white vertices 1, 3, 4, 6, while $b$ is adjacent to 1, 2, 4, 5 and $c$ to 2, 3, 5, 6.

	It is interesting to see this construction as a special case of two previous constructions.  First it is made from $\MG(\M)$ using a $\BGCG$ construction in which each edge of $\MG(\M)$ is paired with the one opposite it at the vertex of $\M$ containing the corresponding corner.

	Secondly, it is $\PP$ of an LR structure called a {\em locally dihedral cycle structure} as outlined at the end of \cite{LRcomb}.

  It is clear that if $\M$ is reflexible, then $\XI(\M)$ is edge-transitive; it is surprising, though, that sometimes (criteria still unknown)  $\XI$ of a chiral map can also be edge-transitive.

\section{Sporadic graphs}
\label{sc:Sporadic}

There are a few graphs in the Census which are given familiar names rather than a parametric form.  These are:  $K_5 = C_5(1,2)$, the Octahedron = $K_{2, 2, 2}$.  Also there is the graph ${\rm Odd}(4)$;  its vertices are subsets of $\{1, 2, 3, 4, 5, 6, 7\}$ of size 3.  Two are joned by an edge when the sets are disjoint.  Finally, there is the graph denoted ${\rm Gray}(4)$ due to a construction by Bouwer \cite{B} which generalizes the Gray graph to make a semisymmetric graph of valence $n$ on $2n^n$ vertices, in this case, 512.  In fact we wanted to extend the Census to 512 vertices in order to include this graph.

Out of more than 7000 graphs in this Census, 400 of them have as their listed names one of the tags $\AT[n, i], \HT[n,i]$ or $\SSS[n,i]$ from the computer-generated censi.  These graphs, then, must have no other known constructions.    Each of these might be truly sporadic or, perhaps, might belong to some interesting family not yet recognized.

	Each of these is a research project in its own right, a single example  waiting to be meaningfully generalized.

\section{Open questions}
\label{sc:OQ}
{\begin{enumerate}
\item 
In compiling the Census, whenever we wanted to include a parameterized family ($C_n(1,a)$ (section \ref{sc:Circulants}), $\BC_n(a, b, c, d)$ (section \ref{sc:Bicirculants}), $\Pr_n(a, b, c, d)$ (section \ref{sc:Propellors}), etc. ), it was very helpful to have some established theorems which either completely classified which values of the paraeters gave edge-transitive graphs or restricted those parameters in some way.  In families for which no such  theorems were known, we were forced into brute-force searches, trying all possible values of the parameters.  In many cases this caused our computers to run out of time or space before finishing the search.  There are many families for which no such results or only partial results exist.   So, our first and most pressing question is :\\For which values of parameters are the following graphs edge-transitive (or LR): 

$\AMC(k,n,M)$ (section \ref{sc:Attebery}), \\
$\MSZ(m,n;k, r)$ (section \ref{MSZ}), \\
$\MC3(m,n,a,b,r,t,c))$ (section \ref{MC3}),\\ 
$\LoPr_n(a, b, c,d, e$ (section \ref{sc:OD}), \\
$\WH_n(a, b, c, d)$ (section \ref{sc:OD}), \\
$\KE_n(a, b, c, d, e)$ (section \ref{sc:OD}), \\
$\Curtain_n(a, b, c, d, e)$ (section \ref{sc:OD}),\\ 
$\CPM(n, s, t, r)$ (section \ref{sc:GP}).

\item  The general class IV metacirculants (of which the graphs $\MSZ$ are merely a part) is still largely unexplored.  First we need to decide how the class is to be parameterized.  Then the same questions as above are relevant: When are these graphs dart-transitive? $\half$-arc-transitive?  LR structures?

\item Given $a$ and $n$ with $a^2\equiv\pm 1$ (mod $n$), which toroidal graph is isomorphic to $\C_n(1, a)$? See sections \ref{sc:Circulants} and \ref{sc:Toroidal Graphs}.  The toroidal graphs are very common and almost every family includes some as special cases.  In researching a family, we often point out that certain values of the parameters give toroidal graphs and so will not be studied with this family.  However, it is often difficult to say exactly which toroidal graph is given by the indicated parameters.  This is simply the first of many such questions.

\item  Under what conditions on their parameters can two spidergraphs be isomorphic?
This is a question mentioned in \cite{STG}.  Many examples of isomorphism theorems are given there as well as examples which show that not all isomorphisms have been found.

\item The Attebery construction presents many challenges.  First of these is the question of isomorphism:  How can $\AMC(k,n,M)$ be isomorphic to $\AMC(k', n', M')$?
This is an even more urgent question, as there are many sets of parameters which give any one isomorphism class of graphs.
 
\item Some graphs with the same diagrams as the metacirculants $\PS, \MPS, \MSY,  \MSZ, \MC3$ are not themselves metacirculants but are nevertheless edge-transitive.    For example consider the graph $\KE_{12}(1, 3, 8, 5, 1)$  It is isomorphic to the graph whose vertices are $\ZZ_4\times\ZZ_{12}$, with each $(1, i)$ adjacent to $(2, i)$ and $(2, i+1)$, each $(2, i)$  to $(3, i)$ and $(3, i+4)$, each $(3, i)$  to $(4, i+4)$ and $(4, i+5)$,   and each $(4, i)$  to $(1, i)$ and $(1, i+10)$.  This diagram is the same `sausage graph' that characterizes the $\PS$ and $\MPS$ graphs and yet the graph is not isomorphic to any $\PS$ or $\MPS$ graph.
This happens rarely enough that the exceptional cases might be classifiable.

\item How many non-isomorphic semitransitive orientations for the graphs with large vertex stabilizers, i.e., the Praeger-Xu graphs, including $\W(n, 2)$  and  $ \RW_{2m}(m+2, m+1)$?  More generally, these graphs present many problems computationally.  Can we determine without computers their cycle structures, useful colorings for BGCG constructions, and other properties?

\item  When do the constructions DG, HC, TAG, applied to some cubic graph $\Lambda$ or some rotary map $\M$, simply result in the line graph or medial graph of some larger  graph or map?

\item The BGCG constructions we have used here are only the beginning of this topic.   We have some constructions for cases where the connection graph is $K_1, K_2$, or $C_k$, but for other connection graphs, we have no general techniques at all.

\item  How can $\XI$ of a chiral map be edge-transitive?

\item The 3-arc graph of a cubic graph (see section \ref{ssc:TAG}) is the partial line graph of some cycle decomposition;  {\em what} decomposition?

\item In section \ref{sc:SomeLR}, we use cycle structures to construct LR structures.  What is an efficient way to find all isomorphism classes of cycles structures for a given dart-transitive tetravalent graph?

\end{enumerate}


\begin{thebibliography}{99}

\bibitem{MSY} I. Anton\v ci\v c, S.Wilson, Symmetries of the Maru\v si\v c-\v Sparl Y graphs, in preparation.

\bibitem{CAthesis}   C.\ Attebery, {\em Constructing Graded Semi-Transitive Orientations of valency $4(p-1)$}, 
M.Sc. Thesis, Northern Arizona University, 2008.

\bibitem{BC}
N. Biggs,  {\em Aspects of symmetry in graphs}, Algebraic methods in graph theory, Vol. I, II (Szeged, 1978),  
pp. 27--35, Colloq. Math. Soc. J\'anos Bolyai, 25, North-Holland, Amsterdam-New York,  1981. 

\bibitem{ConCycHAT} M.\ Boben, \v{S}.\ Miklavi\v{c}, P.\ Poto\v{c}nik,
Consistent cycles in $\frac{1}{2}$-arc-transitive graphs,
      {\em  Electronic J.\ Combin.} {\bf 16} (2009), $\#$ R5, 1--10.
      
\bibitem{magma}  W.~Bosma, J.~Cannon and C.~Playoust, The {\sc Magma} Algebra System I: The User Language, {\em J.~Symbolic Comput.} \textbf{24} (1997), 235–265.

\bibitem{B}   I. Z.\ Bouwer, On edge but not vertex transitive regular graphs, 
    {\em J.\ Combin. Theory Ser.\ B} {\bf 12} (1972), 32--40.

\bibitem{ConDob}
M.\ D.\ E.\ Conder and P.\ Dobcs\' anyi,
Trivalent symmetric graphs on up to $768$ vertices,
{\em J.\ Combin.\ Math.\ Combin.\ Comput.} {\bf 40} (2002), 41--63.

\bibitem{MC3} M. Conder,  
{\em Trivalent (cubic) symmetric graphs on up to 10000 vertices}, 
\href{https://www.math.auckland.ac.nz/~conder/symmcubic10000list.txt}{https://www.math.auckland.ac.nz/$\sim$conder/symmcubic10000list.txt}, accessed 20th January 2016.

\bibitem{CMMP}
M.\ D.\ E.\ Conder, A.\ Malni\v{c}, D.\ Maru\v{s}i\v{c}, P.\ Poto\v{c}nik, A census of semisymmetric cubic graphs on up
to $768$ vertices, {\em J.\ Alg.\ Combin.} {\bf 23} (2006),  255--294.

\bibitem{djo} D.~\v{Z}.~Djokovi\'{c}, A class of finite group-amalgams, 
  {\em Proc. American Math. Soc.} \textbf{80} (1980), 22--26. 

\bibitem{FH} D.\ Firth, \emph{An algorithm to find normal subgroups of a finitely presented group up to a given index}, PhD Thesis, University of Warwick, 2005.

\bibitem{F} J.~Folkman, Regular line-symmetric graphs,
              {\em J. Combin. Theory} {\bf 3} (1967), 215--232.

\bibitem{GP} A.\ Gardiner, C.\ Praeger, A characterization of Certain Families of 4-Valent Symmetric Graphs, 
{\em Europ. J. Combinatorics} {\bf 15} (1994), 383--397.

\bibitem{HW} A.\ Hill, S.\ Wilson, Four Constructions of Highly symmetric Graphs, 
  {\em J. Graph Theory} {\bf 71} (2012),  229--244.

\bibitem{KZ} M.\ Knor, S.\ Zhou, Diameter and connectivity of 3-arc graphs, 
{\em Discrete Math.} {\bf 310} (2010), 37--42.

\bibitem{CircAT1}   I.\ Kovacs, Classifying arc-transitive circulants, {\em J.\ Algebraic Combin.} {\bf 20} (2004), 353–358.

\bibitem{KKM}       I.\ Kovacs, K.\ Kutnar, D.\ Maru\v si\v c,   Classification of 
              edge-transitive rose window graphs, Journal of Graph Theory
              {\bf 65},  3(2010), pp. 216-231.

\bibitem{KKMW}   I.\ Kovacs, B.\ Kuzman, A. Malni\v c, S.\ Wilson, Characterization of edge-transitive 4-valent bicirculants, 
{\em J.\ Graph Theory} {\bf 66}  (2011), 441--463.

\bibitem{Lantz}  B.\ Lantz, {\em Symmetries of Tetravalent Metacirculant Graphs of type III},
      M.Sc. Thesis, Northern Arizona University, 2014.

\bibitem{CircAT2}   C.-H.\ Li, Permutation groups with a cyclic regular subgroup and arc-transitive circulants, 
{\em J.\ Algebraic Combin.} {\bf 21} (2005), 131--136.

\bibitem{MNS} A.\ Malni\v{c}, R.\ Nedela, M.\ \v{S}koviera,      
   Lifting graph automorphisms by voltage assignments,
                 {\em Europ.~J.~Combin.} {\bf 21} (2000), 
                                 927--947. 

\bibitem{M1} D.\ Maru\v si\v c   $\half$-transitive group actions on finite graphs of valency 4, 
  {\em J.\ Combin.\ Theory Ser. B} {\bf 73} (1998) 41--76.

\bibitem{MarNed}
D.~Maru\v{s}i\v{c} and R.~Nedela,
On the point stabilizers of transitive groups with non-self-paired suborbits of length 2,
{\em J.\ Group Theory} {\bf 4} (2001), 19--43.

\bibitem{MP} D.\ Maru\v si\v c and P.\ Poto\v cnik,
                Bridging semisymmetric and $\half$-arc-transitive actions on graphs,
                {\em European J.\ Combin.} {\bf 23} (2002), 719--732.
                
\bibitem{MS} D. Maru\v si\v c, P. \v{S}parl, On quartic $\half$-arc-transitive metacirculants, 
   {\em J.\ Alg. Combin.} {\bf 28} (2008) 365--395.

\bibitem{ConCyc}
\v S.\ Miklavi\v c, P.\ Poto\v cnik and S.\ Wilson, Consistent cycles in graphs and digraphs,
{\em Graphs Combin.} {\bf 23} (2007), 205--216.
 
 \bibitem{AT2} P.~Poto\v{c}nik, A list of $4$-valent $2$-arc-transitive graphs and finite faithful amalgams of index $(4,2)$,
   {\em European J.\ Combin.} \textbf{30} (2009), 1323--1336.

\bibitem{gamma}  P.\ Poto\v{c}nik, P.\ Spiga, G.\ Verret,
    Tetravalent arc-transitive graphs with unbounded vertex-stabilisers,
    {\em   Bull.\ Australian Math.\ Soc.}  {\bf 84} (2011), 79--89.
             
\bibitem{CubicCensus}
P.\ Poto\v{c}nik, P.\ Spiga, G.\ Verret,
       Cubic vertex-transitive graphs on up to $1280$ vertices,     
    {\em J.\ Symbolic Comp.}  {\bf 50} (2013) 465--477.

\bibitem{CubicCensusSite}
P.\ Poto\v{c}nik, P.\ Spiga, G.\ Verret,
{\em A census of small connected cubic vertex-transitive graphs},
\href{http://www.matapp.unimib.it/~spiga/census.html}{http://www.matapp.unimib.it/$\sim$spiga/census.html},
accessed January 20th 2016.

\bibitem{lost}  P.\ Poto\v{c}nik, P.\ Spiga, G.\ Verret,
Bounding the order of the vertex-stabiliser in $3$-valent vertex-transitive and $4$-valent arc-transitive graphs,     
     {\em  J.\ Combin.\ Theory, Ser. B.} {\bf 111} (2015), 148--180.
    
\bibitem{HATcensus}  P.\ Poto\v{c}nik, P.\ Spiga, G.\ Verret,
 A census of $4$-valent $\half$-arc-transitive graphs and arc-transitive digraphs of valence two,
{\em Ars Math.\ Contemp.} {\bf 8} (2015), 133--148.

\bibitem{HATcensusSite}  P.\ Poto\v cnik, P.\ Spiga, G.\ Verret, S.\ Wilson,
{\em Lists of tetravalent arc-transitive, $\half$-arcs-transitive and semisymmetric graphs},
\href{http://www.fmf.uni-lj.si/~potocnik/work.htm}{http://www.fmf.uni-lj.si/$\sim$potocnik/work.htm},
accessed January 20th 2016.

\bibitem{g34} P.\ Poto\v{c}nik and S.~Wilson, Tetravalent edge-transitive 
graphs of girth at most $4$, {\em J.\ Comb.\ Theory, Ser.\ B} {\bf 97} (2007),  217--236.

\bibitem{PWCD} P.\ Poto\v{c}nik and S.~Wilson, Arc-transitive cycle decompositions of tetravalent graphs, 
{\em J.\ Combin. Thoery Ser.\ B} {\bf 98} (2008), 1181--1192.

\bibitem{LR1} P.\ Poto\v{c}nik and S.~Wilson, Linking Rings Structures and tetravalent semisymmetric graphs, 
 {\em Ars Math.\ Contemp.}  {\bf 7} (2014), 341--352

\bibitem{LRalg} P.\ Poto\v{c}nik and S.~Wilson, Linking Rings Structures and Semisymmetric Graphs: Cayley Constructions, {\em Europ.\ J.\ Combin.} {\bf 51} (2016),  84--98

\bibitem{LRcomb} P.\ Poto\v{c}nik and S.~Wilson, Linking Rings Structures and Semisymmetric Graphs: Combinatorial Constructions,  submitted.

\bibitem{SepBox} P.\ Poto\v{c}nik and S.~Wilson, The Separated Box Product of Two Digraphs, submitted.

\bibitem{PX}  C. Praeger, M.-Y. Xu. A Characterization of a Class of Symmetric Graphs of Twice Prime Valency, 
{\em Europ. J.\ Combin.} {\bf 10} (1989), 91--102.

\bibitem{genlost} P.~Spiga, G.~Verret,  On the order of vertex-stabilisers in vertex-transitive graphs with local group $\mathrm{C}_p\times\mathrm{C}_p$ or $\mathrm{C}_p \wr \mathrm{C}_2$,  \href{http://arxiv.org/pdf/1311.4308.pdf}{arXiv:1311.4308 [math.CO]}.

\bibitem{STW} J. \v Siran,T.W. Tucker and M.E. Watkins, Realizing finite edge-transitive orientable maps, 
{\em J.\ Graph Theory} {\bf 37} (2001) 1--34.

\bibitem {S1} P. \v Sparl, A classification of tightly attached $\half$-arc-transitive
graphs of valency 4, 
{\em J. Combin. Theory Ser.\ B} {\bf 98} (2008) 1076--1108.

\bibitem{MSthesis}  Sterns, M.,  Symmetries of Propellor graphs, M.Sc. Thesis, Northern Arizona University, 2008.

\bibitem{Pr}  Sterns, M.,  Classification of Edge-transitive Propellor graphs, submitted, 2015.

\bibitem{tutte} W.~T.~Tutte,  \emph{Connectivity in graphs}, University of Toronto Press, Toronto, 1966.

\bibitem{weiss}
R.\ Weiss,
Presentation for $(G,s)$-transitive graphs of small valency,
{\em Math.\ Proc.\ Phil.\ Soc.} {\bf 101} (1987), 7--20.

\bibitem{maps} S. Wilson, Uniform maps on the Klein bottle,
{\em J. Geom.\ and Graphics} {\bf 10} (2006), 161--171.

\bibitem{STG} S. Wilson, Semi-Transitive Graphs, {\em J.\ Graph Theory} {\bf 45} (2004), 1--27.

\bibitem{BGCG} S. Wilson, P. Poto\v cnik, G, Verret, {\em Base Graph-Connection Graph: Dissection and  Constructions},
in preparation.


\end{thebibliography}
\end{document}